\documentclass[11pt]{amsart}
\usepackage{amsmath, amsthm, amssymb}
\usepackage{graphicx}
\usepackage{mathbbol}
\usepackage{txfonts}
\usepackage[usenames]{color}
\usepackage{srcltx} % JD - to jump between source and preview

\newtheorem{thm}{Theorem}[section]
\newcommand{\bt}{\begin{thm}}
\newcommand{\et}{\end{thm}}

\newtheorem{conj}[thm]{Conjecture}

\newtheorem{cor}[thm]{Corollary}   %remember switch all {coro} to {cor}
\newcommand{\bc}{\begin{cor}}
\newcommand{\ec}{\end{cor}}

\newtheorem{lem}[thm]{Lemma}   %remember to switch all {lem} to {lem}
\newcommand{\bl}{\begin{lem}}
\newcommand{\el}{\end{lem}}

\newtheorem{prop}[thm]{Proposition}
\newcommand{\bp}{\begin{prop}}
\newcommand{\ep}{\end{prop}}

\newtheorem{defn}[thm]{Definition}
\newcommand{\bd}{\begin{defn}}    % This produces an error????    
\newcommand{\ed}{\end{defn}}

\newtheorem{rmrk}[thm]{Remark}   %remember to switch all {rmrk} to {rmrk}

\newcommand{\br}{\begin{rmrk}}
\newcommand{\er}{\end{rmrk}}

\newtheorem{example}[thm]{Example}

\newcommand{\mGHto}{\stackrel { \textrm{mGH}}{\longrightarrow} }
\newcommand{\GHto}{\stackrel { \textrm{GH}}{\longrightarrow} }
\newcommand{\Fto}{\stackrel {\mathcal{F}}{\longrightarrow} }

\newcommand{\be}{\begin{equation}}

\newcommand{\ee}{\end{equation}}

\newcommand{\N}{\mathbb{N}}

\newcommand{\R}{\mathbb{R}}
\newcommand{\One}{{\bf \rm{1}}}

\newcommand{\E}{\mathbb{E}}

\newcommand{\dist}{\operatorname{dist}}
\newcommand{\diam}{\operatorname{Diam}}

\newcommand{\Fm}{{\mathcal F}}

\newcommand{\set}{\rm{set}}

\newcommand{\Scal}{{\rm Scal}} % new Jan 2016
\newcommand{\disjointunion}{\sqcup}

\newcommand{\Lip}{\operatorname{Lip}}

\newcommand{\mass}{{\mathbf M}}

\newcommand{\intcurr}{{\mathbf I}}      %metric integral current%

\newcommand{\vol}{\operatorname{Vol}}

\newcommand{\rstr}{\:\mbox{\rule{0.1ex}{1.2ex}\rule{1.1ex}{0.1ex}}\:}
\newcommand{\bdry}{\partial}

\begin{document}

\title{Sewing Riemannian Manifolds with Positive Scalar Curvature}

\author{J. Basilio}
\thanks{J. Basilio was partially supported as a doctoral student by NSF DMS 1006059.}
\address{CUNY Graduate Center and Sarah Lawrence College}
\email{jorge.math.basilio@gmail.com}

\author{J. Dodziuk}
\address{CUNY Graduate Center and Queens College}
\email{jdodziuk@gmail.com}

\author{C. Sormani}
\thanks{C. Sormani was partially supported by NSF DMS 1006059.}
\address{CUNY Graduate Center and Lehman College}
\email{sormanic@gmail.com}

%\date{February 2017}


%49Q15 (Geometric measure and integration theory, integral and normal currents)

%\subjclass[2000]{49Q15}

\begin{abstract}
We explore to what extent one may hope to preserve geometric
properties of three dimensional manifolds with
lower scalar curvature bounds under Gromov-Hausdorff and
Intrinsic Flat limits.   We introduce a new construction, called sewing, of three
dimensional manifolds that preserves positive scalar curvature.
We then use sewing to produce sequences of such manifolds which converge to spaces that fail to have  nonnegative scalar curvature in a standard  
generalized sense.   Since the notion of nonnegative scalar
curvature is not strong enough to persist alone, we propose that
one pair a lower scalar curvature bound with a lower bound on the area of
a closed minimal surface when taking sequences as this will exclude the possibility
of sewing of manifolds.   
\end{abstract}

\maketitle

\section{Introduction}

In this paper we study three dimensional manifolds with positive scalar curvature.
The scalar curvature of a Riemannian manifold is the average of the Ricci curvatures which in turn is the average of the sectional curvatures.   It can be
determined more simply by taking the following limit:
\be \label{sphere-vol}
\Scal(p)= \lim_{r\to 0} 30 \frac{\vol_{\E^3}(B(0,r))-\vol_{M^3}(B(p, r))}{r^2\vol_{\E^3}(B(0,r))}
\ee
where $\vol_{\E^3}(B(0,r))=(4/3) \pi r^3$ and $\vol_{M^3}(B(p, r))$ is the
Hausdorff measure of the ball about $p$ of radius $r$ in our manifold, $M^3$.

In \cite{Gromov-Plateau}, Gromov asks the following pair of deliberately vague questions which we paraphrase here:
{\em Given a class of Riemannian manifolds, $\mathcal{B}$, 
what is the weakest notion of convergence such that a
sequence of manifolds, $M_j \in \mathcal{B}$,  
subconverges to a limit $M_\infty \in \mathcal{B}$
where now we will expand $\mathcal{B}$ to include singular metric spaces?  
What is this generalized class of singular metrics spaces that should be
included in $\mathcal{B}$? } Gromov points out that when $\mathcal{B}$ is
the class of Riemannian manifolds with nonnegative sectional curvature then the
``best known'' answer to this question is Gromov-Hausdorff convergence
and the singular limit spaces are then Alexandrov spaces with nonnegative Alexandrov curvature.  When $\mathcal{B}$ is the class of Riemannian manifolds with nonnegative Ricci curvature, one uses Gromov-Hausdorff and metric measure convergence to obtain limits which are metric measure spaces with generalized nonnegative Ricci curvature as in work of Cheeger-Colding \cite{ChCo-PartI}.  Work towards defining classes of singular metric measure spaces with generalized notions of nonnegative Ricci has been completed by Lott-Villani, Sturm, Ambrosio-Gigli-Savare and others
\cite{Lott-Villani-09} \cite{Sturm-curv} \cite{AGS}.  

Gromov then writes that {\em ``the most tantalizing 
relation $\mathcal{B}$ is expressed with the scalar curvature by $\Scal \ge k$''} \cite{Gromov-Plateau}.   Bamler \cite{Bamler-16} and Gromov \cite{Gromov-Dirac} have proven that under $C^0$ convergence to smooth Riemannian limits $\Scal \ge 0$ is preserved.   In order to find the weakest notion of convergence which preserves $\Scal \ge 0$ in some sense, Gromov has suggested that one might investigate intrinsic flat convergence \cite{Gromov-Plateau}.   
The intrinsic flat distance
was first defined in work of the third author with Wenger \cite{SW-JDG}, who also proved that for noncollapsing sequences of manifolds with nonnegative Ricci curvature, intrinsic flat limits agree with Gromov-Hausdorff and metric measure limits \cite{SW-CVPDE}.   Intrinsic flat convergence is a weaker notion of convergence in the sense that there are sequences of manifolds with no Gromov-Hausdorff limit that have intrinsic flat limits, including Ilmanen's Example of a sequence of three spheres with positive scalar curvature \cite{SW-JDG}.  The third author has investigated intrinsic flat limits of manifolds with nonnegative scalar curvature under additional conditions with Lee, Huang, LeFloch and Stavrov \cite{LeeSormani1}\cite{HLS}\cite{LeFloch-Sormani-1} \cite{Sormani-Stavrov-1}.  These papers support Gromov's suggestion in the sense that the limits obtained in these papers have generalized nonnegative scalar curvature.   

Here we construct a sequence of Riemannian manifolds, $M_j^{3}$, with positive
scalar curvature that converges in the intrinsic flat, metric measure and Gromov-Hausdorff sense to a singular limit space, $Y$, 
which fails to satisfy (\ref{sphere-vol}) [Example~\ref{sphere-geod}].    In fact,
the limit space is a sphere with a pulled thread:
\be 
Y= {\mathbb S}^3 / \sim \textrm{ where } a\sim b \textrm{ iff } a,b \in C,
\ee
where $C$ is one geodesic in $\mathbb{S}^3$ (see Section~\ref{sect-pulled-string}).
The scalar curvature
about the point $p_0 =[C(t)]$ formed from the pulled thread is computed 
in Lemma~\ref{bad-point} to be
\be\label{sphere-vol-bad}
\lim_{r\to 0} \,\,\frac{\,\,\vol_{\E^3}(B(0,r))-\vol_{M^3}(B(p, r))\,\,}{r^2\vol_{\E^3}(B(0,r))}\,\,=\,\,-\infty.
\ee
In this sense the limit space does not have generalized nonnegative scalar curvature.

We construct our sequence using a new method we call sewing 
developed in Propositions~\ref{prop-glue}-\ref{sewn-curve}.  Before we can sew the manifolds, the first two authors construct short tunnels
between points in the manifolds building on prior work of Gromov-Lawson and
Schoen-Yau in
\cite{Gromov-Lawson-tunnels} \cite{Schoen-Yau-tunnels}.  The details of this construction are in the Appendix.   In a subsequent paper \cite{Basilio-Sormani-1} we will extend this sewing technique to also provide examples whose limit spaces fail to satisfy the Scalar Torus Rigidity Theorem \cite{Schoen-Yau-tunnels} \cite{Gromov-Lawson-tunnels} and the Positive Mass Rigidity Theorem \cite{Schoen-Yau-positive-mass}.   These examples, all constructed using the sewing techniques developed in this paper, demonstrate that Gromov-Hausdorff and Intrinsic Flat limit spaces of noncollapsing sequences of manifolds with
positive scalar curvature may fail to satisfy key properties of nonnegative scalar curvature. 

In light of these counter examples and the aforementioned positive results towards Gromov's conjecture, the third author has suggested in \cite{Sormani-scalar} to adapt the class $\mathcal{B}$.   There it is proposed that the 
initial class of smooth Riemannian manifolds in $\mathcal{B}$ should have
nonnegative scalar curvature, a uniform lower bound on volume (as assumed implicitly by Gromov), and also a uniform lower bound on the minimal area of a closed minimal surface in the manifold, $\textrm{MinA}(M)$.  
The sequences 
of $M_j^{3}$ we construct using our new sewing methods have positive scalar curvature and a uniform
lower bound on volume, but $\textrm{MinA}(M_j)\to 0$.  Intuitive reasons as to
why a uniform lower bound on  
$\textrm{MinA}(M_j)$ is a natural condition are described in \cite{Sormani-scalar}
along with a collection of related conjectures and open problems.
Here we will simply propose
the following possible revision of Gromov's vague conjecture:  

\begin{conj}
Suppose a sequence of Riemannian manifolds, $M^3_j$, have 
\be
\Scal_j\ge 0, \vol(M_j)\ge V_0>0, \textrm{ and }
\textrm{MinA}(M_j)\ge A_0>0,   
\ee
and the sequence converges in the intrinsic flat sense, $M_j \Fto M_\infty$. 

Then at every point $p\in M_\infty$ we have
\be\label{sphere-vol-good}
\lim_{r\to 0} \,\,\frac{\,\,\vol_{\E^3}(B(0,r))-\vol_{Y}(B(p, r))\,\,}{r^2\vol_{\E^3}(B(0,r))}\,\,\ge \,\, 0.
\ee
\end{conj}

This paper is part of the work towards Jorge Basilio's doctoral dissertation at the
CUNY Graduate Center conducted under the advisement of Professors J\'ozef Dodziuk and Christina Sormani.   We would like to thank Jeff Jauregui, Marcus Khuri, Sajjad Lakzian, Dan Lee, Raquel Perales, Conrad Plaut, 
and Catherine Searle      
%maybe add Dan King and Philip Ording and others who invited speakers
for their interest in this work.

%%%%%%%%%%%%%%%%%%

\section{Background}\label{S:background}

In this section we first briefly review Gromov-Lawson and Schoen-Yau's work.
We then review Gromov-Hausdorff, Metric Measure, and Intrinsic Flat
Convergence covering the key definitions as well as theorems applied in this
paper to prove our example converges with respect to all three notions of
convergence.

\subsection{Gluing Gromov-Lawson and Schoen-Yau tunnels}\label{sect-tunnel-back}

%We ought to explain a bit about the Schoen-Yau construction vs the Gromov-Lawson one.

Using different techniques, Gromov-Lawson and Schoen-Yau described how to
construct tunnels diffeomorphic to ${\mathbb{S}}^2 \times [0,1]$
with metric tensors of positive scalar curvature that can be glued smoothly
into three dimensional spheres of constant sectional curvature \cite{Gromov-Lawson-tunnels}\cite{Schoen-Yau-tunnels}.   See Figure~\ref{Fig.SY-GL-tunnel}.
These tunnels are the first crucial piece for our construction.

\begin{figure}[h]
\includegraphics[scale=0.15]{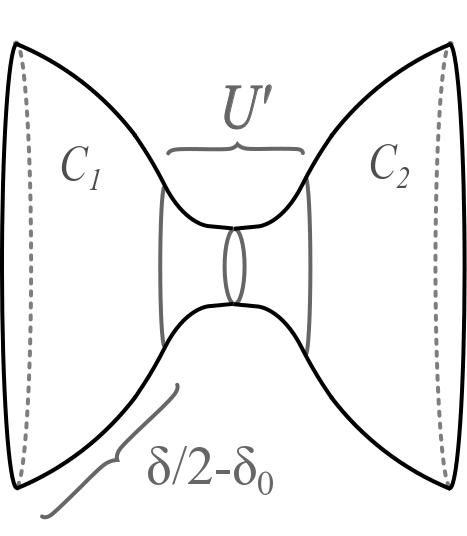}
\caption{The Tunnel}
\label{Fig.SY-GL-tunnel}
\end{figure}

Here we need to explicitly estimate the volume and diameter of these tunnels.   So the first and second authors prove the following
lemma in the appendix.

% BEGIN LEMMA STATEMENT
\begin{lem}%(\cite{Gromov-Lawson-tunnels}\cite{Schoen-Yau-tunnels})
\label{tunnellemma}
	Let $0<\delta/2 < 1$. % ADDED 3-31-16 
	Given a complete Riemannian manifold, $M^3$, 
	that contains two balls $B(p_i,\delta/2)\subset M^3$, $i=1,2$, with constant positive sectional curvature $K \in (0,1]$ on the balls, 
	and given any $\epsilon>0$, there exists a $\delta_0>0$ sufficiently small so that we may create a new
	complete Riemannian manifold, $N^3$, 
	in which we remove two balls and glue in a cylindrical region, $U$, between them:
	\be\label{TL-sewnN}
	N^3=M^3 \setminus \left(B(p_1,\delta/2)\cup B(p_2,\delta/2)\right) \disjointunion U
	\ee
	where $U=U(\delta_0)$ has a metric of positive scalar curvature (See Figure~\ref{Fig.SY-GL-tunnel}) with
	\be\label{TL-diameterU}
	\diam(U) \le h=h(\delta), 
	%=\left(\frac{\pi}{2}+\frac{1}{2}\right)\delta + \left(\frac{2}{M}+\frac{K}{M}\right)\delta_0 + \left(4+\frac{K}{M}\right)\sigma % CLEAN UP FURTHER % CHECK GOES TO ZERO WITH DELTA TO 0
	\ee
	where
	\be\label{TL-diameterUorderdelta}
	h(\delta)=O(\delta),
	\ee
	hence,
	\be\label{TL-tunnellengthtozero}
	\lim_{\delta\to 0} h(\delta)=0 %\quad (\textrm{or }O(\delta)) 
	\textrm{ uniformly for } K\in (0,1].
	\ee
	The collars $C_i= B(p_i,\delta/2) \setminus B(p_i,\delta_0)$ identified with subsets of
	$N^3$ have the original metric of constant curvature and the tunnel $U'=U\setminus (C_1\cup C_2)$ has arbitrarily small diameter $O(\delta_0)$ and volume $O(\delta_0^3)$. 
	Therefore with appropriate choice of $\delta_0$, we have
	\be\label{TL-volumeestU}
	(1-\epsilon) 2\vol(B(p,\delta/2)) \le \vol(U) \le (1+\epsilon) 2\vol(B(p,\delta/2))
	\ee
	and
	\be\label{TL-volumeestN}
	(1-\epsilon) \vol(M) \le \vol(N) \le (1+\epsilon) \vol(M).
	\ee
	
\end{lem}

We note that if $M^3$ has positive scalar curvature then so does $N^3$ and that, after inserting the tunnel,
$\partial B(p_1,\delta/2)$ and $\partial B(p_2,\delta/2)$ are arbitrarily close together because of (\ref{TL-tunnellengthtozero}).   Note that we have restricted to three dimensions here and required constant sectional curvature on the balls for simplicity.  The first two authors will generalize these conditions in future work.   This lemma suffices for proving all the examples in this paper.

% Review GH conv
\subsection{Review GH Convergence} \label{sect-GH-back}

Gromov introduced the Gromov-Hausdorff distance in \cite{Gromov-metric}.

First recall that $\varphi: X \to Y$ is distance preserving iff
\be
d_Y(\varphi(x_1), \varphi(x_2)) = d_X(x_1, x_2) \qquad \forall x_1, x_2 \in X.
\ee
This is referred to as a metric isometric embedding in \cite{LeeSormani1}
and is distinct from a Riemannian isometric embedding.

\begin{defn}[Gromov]\label{defn-GH} 
The Gromov-Hausdorff distance between two 
compact metric spaces $\left(X, d_X\right)$ and $\left(Y, d_Y\right)$
is defined as
\be \label{eqn-GH-def}
d_{GH}\left(X,Y\right) := \inf  \, d^Z_H\left(\varphi\left(X\right), \psi\left(Y\right)\right)
\ee
where $Z$ is a complete metric space, and $\varphi: X \to Z$ and $\psi:Y\to Z$ are
distance preserving maps and where the Hausdorff distance in $Z$ is defined as
\be
d_{H}^Z\left(A,B\right) = \inf\{ \epsilon>0: A \subset T_\epsilon\left(B\right) \textrm{ and } B \subset T_\epsilon\left(A\right)\}.
\ee
\end{defn}

Gromov proved that this is indeed a distance on compact metric spaces: $d_{GH}\left(X,Y\right)=0$
iff there is an isometry between $X$ and $Y$.   When studying metric
spaces which are only precompact, one may take their metric
completions before studying the Gromov-Hausdorff distance
between them.   

We write 
\be
X_j \GHto X_\infty \,\,\,\textrm{ iff }\,\,\, d_{GH}(X_j, X_\infty) \to 0.
\ee
Gromov proved that if $X_j \GHto X_\infty$ then there is a common
compact metric space $Z$ and distance preserving maps $\varphi_j: X_j \to Z$
such that 
\be \label{common-Z-GH}
d^Z_H(\varphi_j(X_j), \varphi_\infty(X_\infty))\to 0.
\ee
We say $p_j \in X_j$ converges to $p_\infty\in X_\infty$ if there is such a set of maps such that $\varphi_j(p_j)$
converges to $\varphi_\infty(p_\infty)$ as points in $Z$.  These limits are not uniquely defined but they are useful and every point in the limit space is a limit of such a sequence in this sense.  

\begin{thm} [Gromov] \label{almost-isom} %this is used
Suppose $\epsilon_j \to 0$. If a sequence of metric spaces $(X_j, d_j)$ have $\epsilon_j$ almost isometries
\be
F_j: X_j \to X_\infty
\ee
such that
\be
|d_\infty(F_j(p), F_j(q)) - d_j(p,q)| \le \epsilon_j \qquad \forall p,q\in X_j
\ee
and
\be
X_\infty \subset T_{\epsilon_j}(F_j(X_j))
\ee
 then
\be
X_j \GHto X_\infty.
\ee
\end{thm}

Note that $p_j \in X_j$ converges to $p_\infty\in X_\infty$
if $F_j(p_j) \to p_\infty \in X_\infty$.  

Gromov's Compactness Theorem states that a sequence of manifolds with nonnegative Ricci (or Sectional) Curvature, and a uniform upper bound on diameter, has a subsequence which converges in the Gromov-Hausdorff sense to 
a geodesic metric space \cite{Gromov-metric}. If a sequence of manifolds has nonnegative sectional curvature, then they satisfy the Toponogov Triangle Comparison Theorem.  Taking the limits of the points in the triangles, one sees that the Gromov-Hausdorff limit of the sequence also satisfies the triangle comparison.  Thus the limit spaces are Alexandrov spaces with nonnegative Alexandrov curvature (cf. \cite{BBI}).   

\subsection{Review of Metric Measure Convergence} \label{sect-mm-back}

Fukaya introduced the notion of metric measure convergence of metric measure spaces $(X_j, d_j, \mu_j)$ in \cite{Fukaya-1987}.  He assumed the sequence converged in the Gromov-Hausdorff sense as in (\ref{common-Z-GH}) and then required that the push forwards of the measures converge as well,
\be
\varphi_{j*}\mu_j \to \varphi_{\infty*}\mu_\infty  \textrm{ weakly as measures in } Z.
\ee 
Cheeger--Colding proved metric measure convergence of noncollapsing sequences of manifolds with Ricci uniformly bounded below in \cite{ChCo-PartI} where the measure on the limit is the Hausdorff measure.  They proved metric measure convergence by constructing almost isometries and showing the Hausdorff measures of balls about converging points converge:
\be
\textrm{ If } p_j \to p_\infty \textrm{ then } \mathcal{H}^m(B(p_j,r)) \to \mathcal{H}^m(B(p_\infty,r)).
\ee
  They also studied collapsing sequences obtaining metric measure convergence to other measures on the limit space.  Cheeger and Colding applied this metric measure convergence to prove that limits of manifolds with nonnegative Ricci curvature have generalized nonnegative Ricci curvature.  In particular they prove the limits satisfy the Bishop-Gromov Volume Comparison Theorem and the Cheeger-Gromoll Splitting Theorem.  

Sturm, Lott and Villani then developed the CD(k,n) notion of generalized Ricci curvature on metric measure spaces in \cite{Sturm-curv}\cite{Lott-Villani-09}.   In \cite{Sturm-06}, Sturm extended the study of metric measure convergence beyond the consideration of sequences of manifolds which already converge in the Gromov-Hausdorff sense, using the Wasserstein distance.  This is also explored in Villani's text \cite{Villani-text}.   CD(k,n) spaces converge in this sense to CD(k,n) spaces.  RCD(k,n) spaces developed by Ambrosio-Gigli-Savare are also preserved under this convergence \cite{AGS}.  RCD(k,n) spaces are CD(k,n) spaces which also require that the tangent cones almost everywhere are Hilbertian.  There has been significant work studying both of these classes of spaces proving they satisfy many of the properties of Riemannian manifolds with lower bounds on their Ricci curvature.

\vspace{.4cm}
\subsection{Review of Integral Current Spaces} \label{sect-C-back}

The Intrinsic Flat Distance is defined and studied in \cite{SW-JDG} by applying sophisticated ideas of Ambrosio-Kirchheim \cite{AK}
extending earlier work of Federer-Fleming \cite{FF}.   Limits of Riemannian
manifolds under intrinsic flat convergence are integral current spaces, a notion
 introduced by
the third author and Stefan Wenger in \cite{SW-JDG}.  

Recall that Federer-Flemming first defined the
notion of an integral current as an extension of the notion of a submanifold of Euclidean space
\cite{FF}.  That is a submanifold $\psi: M^m \to \mathbb{E}^N$ can be viewed
as a current $T=\psi_{\#}\lbrack M \rbrack$ acting on $m$-forms as follows:
\be
T(\omega)= \psi_{\#}\lbrack M \rbrack(\omega) 
= \lbrack M \rbrack (\psi^*\omega) =\int_M \psi^*\omega.
\ee
If $\omega= f\,d\pi_1 \wedge \cdots \wedge d\pi_m$ then
\be
T(\omega)=\psi_{\#}\lbrack M \rbrack(\omega) =
\int_M f\circ \psi\,d(\pi_1\circ \psi) \wedge \cdots \wedge d(\pi_m\circ \psi).
\ee
They define boundaries of currents as $\partial T(\omega) = T (d\omega)$
so that then the boundary of a submanifold with boundary is exactly what it should be.  They define integer rectifiable currents more generally as countable sums of
images under Lipschitz maps of Borel sets.  The integral currents are
integer rectifiable currents whose boundaries are integer rectifiable.  

Ambrosio-Kirchheim extended the notion of integral currents to arbitrary
complete metric space \cite{AK}.  As there are no forms on metric spaces,
they use deGeorgi's tuples of Lipschitz functions,
\be
T(f, \pi_1,..., \pi_m)=\psi_{\#}\lbrack M \rbrack(f, \pi_1,..., \pi_m)=
\int_M f\circ \psi\,d(\pi_1\circ \psi) \wedge \cdots \wedge d(\pi_m\circ \psi).
\ee
This integral is well defined because Lipschitz functions are differentiable almost everywhere.   They define boundary as follows:
\be
\partial T(f, \pi_1,..., \pi_m)= T(1,f, \pi_1,..., \pi_m)
\ee 
which matches with 
\be
d(f\,d\pi_1 \wedge \cdots \wedge d\pi_m)
=1\,df \wedge d\pi_1 \wedge \cdots \wedge d\pi_m.  
\ee 
They also define integer rectifiable currents more generally as countable sums of
images under Lipschitz maps of Borel sets.  The integral currents are
integer rectifiable currents whose boundaries are integer rectifiable. 

The notion of an integral current space was introduced in \cite{SW-JDG}.

\begin{defn} \label{defn-int-curr-space}
An $m$ dimensional integral current space,
$\left(X,d, T\right)$, is a metric space, $(X,d)$ with an
integral current structure $T \in \intcurr_m\left(\bar{X}\right)$
where $\bar{X}$ is the metric completion of $X$
and $\set(T)=X$.   
Given an 
integral current 
space $M=\left(X,d,T\right)$ we will use
$\set\left(M\right)$ or $X_M$ to denote $X$,  $d_M=d$ and $\Lbrack M \Rbrack =T $. 
Note that $\set\left(\partial T\right) \subset \bar{X}$.   
The boundary of $\left(X,d,T\right)$ is then the integral current space:
\be
\partial \left(X,d_X,T\right) := \left(\set\left(\partial T\right), d_{\bar{X}}, \partial T\right).
\ee
If $\partial T=0$ then we say $\left(X,d,T\right)$ is an integral current without boundary.
\end{defn}

A compact oriented Riemannian manifold with boundary, $M^m$,
is an integral current space, where $X=M^m$, $d$ is the standard
metric on $M$ and $T$ is integration over $M$.  In this
case $\mass(M)=\vol(M)$ and $\partial M$ is the boundary manifold.
When $M$ has no boundary, $\partial M=0$.

Ambrosio-Kirchheim defined the mass $\mass(T)$ and the mass measure
$||T||$ of a current in \cite{AK}.  We apply the same notions to define a mass for an
integral current space.   Applying their theorems we have
\be \label{mass-current-space}
\mass(M)=\mass(T)=\int_X \theta_T(x)\lambda(x) d\mathcal{H}^m(x)
\ee
where $\lambda(x)$ is the area factor and $\theta_T$ is the weight.    
In particular $\lambda(x)=1$ when the
the tangent cone at $x$ is Euclidean which is true on a Riemannian manifold
where the weight is also $1$.  This is true almost everywhere in the examples
in this paper as well. The mass measure, $||T||$, is a measure on $X$ and satisfies
\be \label{mass-measure}
||T||(A)=\int_A \theta_T(x)\lambda(x) d\mathcal{H}^m(x).
\ee

%Rectifiability and other properties.

% Review of IFC
\vspace{.4cm}
\subsection{Review of the Intrinsic Flat distance}\label{sect-SWIF-back}

The Intrinsic Flat distance was defined in work of the third author and
Stefan Wenger \cite{SW-JDG} as a new distance between Riemannian
manifolds based upon the Federer-Flemming flat distance \cite{FF} and
the Gromov-Hausdorff distance \cite{Gromov-metric}.

Recall that the Federer-Flemming flat distance between $m$ dimensional integral currents 
$S,T\in\intcurr_m\left(Z\right)$ is given by 
\begin{equation} \label{eqn-Federer-Flat}
d^Z_{F}\left(S,T\right):= 
\inf\{\mass\left(U\right)+\mass\left(V\right):
S-T=U+\bdry V \}
\end{equation}
where $U\in\intcurr_m\left(Z\right)$ and $V\in\intcurr_{m+1}\left(Z\right)$.

In \cite{SW-JDG}, the third author and Wenger imitate Gromov's definition of the Gromov-Hausdorff distance (which he called the intrinsic Hausdorff distance) by replaced the Hausdorff distance by the Flat distance:

\begin{defn}(\cite{SW-JDG}) \label{def-flat1} 
 For $M_1=\left(X_1,d_1,T_1\right)$ and $M_2=\left(X_2,d_2,T_2\right)\in\mathcal M^m$ let the 
intrinsic flat distance  be defined:
 \begin{equation}\label{equation:def-abstract-flat-distance}
  d_{\Fm}\left(M_1,M_2\right):=
 \inf d_F^Z
\left(\varphi_{1\#} T_1, \varphi_{2\#} T_2 \right),
 \end{equation} 
where the infimum is taken over all complete metric spaces 
$\left(Z,d\right)$ and distance preserving maps 
$\varphi_1 : \left(\bar{X}_1,d_1\right)\to \left(Z,d\right)$ and $\varphi_2: \left(\bar{X}_2,d_2\right)\to \left(Z,d\right)$
and the flat norm $d_F^Z$ is taken in $Z$.
Here $\bar{X}_i$ denotes the metric completion of $X_i$ and $d_i$ is the extension
of $d_i$ on $\bar{X}_i$, while $\varphi_\# T$ denotes the push forward of $T$.
\end{defn}

They then prove that this distance is 0 iff the spaces are isometric with a current preserving isometry.  They say
\be
M_j \Fto M_\infty \textrm{ iff } d_{\mathcal{F}}(M_j, M_\infty) \to 0.
\ee
And prove that this happens iff there is a complete metric space $Z$ and
distance preserving maps $\varphi_j: M_j \to Z$ such that
\be 
d_F^Z(\varphi_{j\#}T_j,\varphi_{\infty\#}T_\infty) \to 0
\ee
Note that in contrast to Gromov's embedding theorem as stated in
(\ref{common-Z-GH}), the $Z$ here is only complete and not compact.

There is a special integral current space called the zero space, 
\be\label{zero-space}
{\bf{0}}=(\emptyset, 0,0).  
\ee
Following the definition above, $M_j \Fto {\bf{0}}$ iff 
$d_{\mathcal{F}}(M_j, {\bf{0}}) \to 0$ which implies 
there is a complete metric space $Z$ and
distance preserving maps $\varphi_j: M_j \to Z$ such that
\be 
d_F^Z(\varphi_{j\#}T_j, 0) \to 0
\ee
Note that in this case the manifolds disappear and points have no limits.

Combining Gromov's Embedding Theorem with Ambrosio-Kitrchheim's Compactness Theorem one has:

\begin{thm}[\cite{SW-JDG}] \label{GH-to-flat}
Given a sequence of $m$ dimensional integral current spaces $M_j=\left(X_j, d_j, T_j\right)$ such that $X_j$ are equibounded and
equicompact and with uniform upper bounds on mass and boundary mass.
A subsequence converges in the
Gromov-Hausdorff sense $\left(X_{j_i}, d_{j_i}\right) \GHto \left(Y,d_Y\right)$ and in the 
intrinsic flat sense 
$\left(X_{j_i}, d_{j_i}, T_{j_i}\right) \Fto \left(X,d,T\right)$
where either $\left(X,d,T\right)$ is an $m$ dimensional integral current space
with $X \subset Y$
or it is the ${\bf 0}$ current space.
\end{thm}

Note that in \cite{SW-CVPDE}, the third author and Wenger prove if the $M_j$  have nonnegative Ricci curvature then in fact the intrinsic flat and Gromov-Hausdorff limits agree.  Matveev and Portegies have extended this to more
general lower bounds on Ricci curvature in \cite{Matveev-Portegies}.  With only lower bounds on scalar curvature the limits need not agree as seen in the
Appendix of \cite{SW-JDG}.  There are also sequences of manifolds with
nonnegative scalar curvatue that have no Gromov-Hausdorff limit but do converge in the intrinsic flat sense (cf. Ilmanen's Example presented in \cite{SW-JDG} and also \cite{Lakzian-Sormani}).

In \cite{Wenger-compactness}, Wenger proved that any sequence of Riemannian
manifolds with a uniform upper bound on diameter, volume and boundary volume has a subsequence which converges in the intrinsic flat sense to an integral current space (cf. \cite{SW-JDG}).  It is possible that the limit space is just the $\bf{0}$ space
which happens for example when the volumes of the manifolds converge to $0$. 

Note that when $M_j \Fto M_\infty$ the masses are lower semicontinuous:
\be \label{mass-semicont}
\liminf_{j\to\infty} \mass(M_j) \ge \mass(M_\infty)
\ee
where the mass of an integral current space is just the mass of the integral
current structure.  The mass is just the volume when $M$ is a Riemannian
manifold and can be computed using (\ref{mass-current-space}) otherwise.
As there is not equality here, intrinsic flat convergence does not imply metric measure convergence. 

In \cite{Portegies-F-evalue}, Portegies has proven that when a sequence converges in the intrinsic flat sense and in addition $\mass(M_j)$ is assumed to converge
to $\mass(M_\infty)$, then the spaces do converge in the metric measure sense,
where the measures are taken to be the mass measures.

\subsection{Useful Lemmas and Theorems concerning Intrinsic Flat convergence}

The following lemmas, definitions and theorems appear in work of the third author
\cite{Sormani-AA}, although a few (labelled only as c.f. \cite{Sormani-AA}) were used
within proofs in older work of the third author with Wenger \cite{SW-CVPDE}.   All are
proven rigorously in \cite{Sormani-AA}.

\begin{lem}(c.f. \cite{Sormani-AA})\label{lem-ball}
A ball in an integral current space, $M=\left(X,d,T\right)$,
with the current restricted from the current structure of the Riemannian manifold is an integral current space itself, 
\be
S\left(p,r\right)=\left(\set(T\rstr B(p,r)),d,T\rstr B\left(p,r\right)\right)
\ee
for almost every $r > 0$.   Furthermore,
\be\label{ball-in-ball}
B(p,r) \subset \set(S(p,r))\subset \bar{B}(p,r)\subset X.
\ee
\end{lem}

\begin{lem} (c.f. \cite{Sormani-AA})
When $M$ is a Riemannian manifold with boundary
\be
S\left(p,r\right)=\left(\bar{B}\left(p,r\right),d,T\rstr B\left(p,r\right)\right)
\ee
is an integral current space for all $r > 0$.
\end{lem}

\begin{defn} \label{point-conv}    (c.f. \cite{Sormani-AA})
If $M_i=(X_i, d_i,T_i) \Fto M_\infty=(X_\infty, d_\infty,T_\infty)$, 
then we say $x_i\in X_i$ are a converging sequence that converge to
$x_\infty\in \bar{X}_\infty$ if there exists a complete metric space
$Z$ and distance preserving maps 
$\varphi_i:X_i\to Z$ such that 
\be
\varphi_{i\#} T_i \Fto \varphi_{\infty\#}T_\infty \textrm{ and }
\varphi_i(x_i) \to \varphi_\infty(x_\infty).
\ee   If we say collection of
points, $\{p_{1,i}, p_{2,i},...p_{k,i}\}$,
converges to a corresponding collection of points, 
$\{p_{1,\infty}, p_{2,\infty},...p_{k,\infty}\}$, if 
$\varphi_{i}(p_{j,i}) \to \varphi_\infty(p_{j, \infty})$ for $j=1,\ldots,k$.
\end{defn}

\begin{defn} (c.f. \cite{Sormani-AA})\label{point-Cauchy}
If $M_i=(X_i, d_i,T_i) \Fto M_\infty=(X_\infty, d_\infty,T_\infty)$, then we say $x_i\in X_i$ 
are Cauchy if there exists a complete metric space
$Z$ and distance preserving maps 
$\varphi_i:M_i\to Z$ such that 
\be
\varphi_{i\#} T_i \Fto \varphi_{\infty\#}T_\infty \textrm{ and }
\varphi_i(x_i) \to z_\infty \in Z.
\ee   We say the
sequence is disappearing if $z_\infty \notin \varphi_\infty(X_\infty)$.
We say the sequence has no limit in $\bar{X}_\infty$ if
$z_\infty \notin \varphi_\infty(\bar{X}_\infty)$.
\end{defn}

\begin{lem}(c.f. \cite{Sormani-AA}) \label{to-a-limit}
If a sequence of integral current spaces, $M_i=\left(X_i,d_i,T_i\right)\in \mathcal{M}_0^m$, 
converges to 
an integral current space, $M=\left(X,d,T\right)\in \mathcal{M}_0^m$, in the intrinsic flat sense, then every point $x$ in the limit space
$X$ is the limit of points $x_i\in M_i$.  
In fact there exists a sequence of maps $F_i: X \to X_i$
such that $x_i=F_i(x)$ converges to $x$ and
\be
\lim_{i\to \infty} d_i(F_i(x), F_i(y))= d(x,y).
\ee
\end{lem}

\begin{lem}(c.f. \cite{Sormani-AA})\label{balls-converge}
If $M_j \Fto M_\infty$ and $p_j \to p_\infty\in \bar{X}_\infty$, then for almost every
$r_\infty>0$ there exists a subsequence of $M_j$ also denoted
$M_j$  such that 
\be
S(p_j,r_\infty)= \left(\bar{B}\left(p_j,r_\infty\right),d_j,T_j\rstr B\left(p_j,r_\infty\right)\right)
\ee
are integral current spaces for $j\in \{1,2,...,\infty\}$ and we have
\be
S(p_j,r_\infty) \Fto S(p_\infty,r_\infty).
\ee
If $p_j$ are Cauchy with no limit in $\bar{X}_\infty$
then there exists $\delta>0$ such that
for almost every $r \in (0,\delta)$ such that
$S(p_j,r)$
are integral current spaces for $j\in \{1,2,...\}$ and we have
\be\label{to-0}
S(p_j,r) \Fto 0.
\ee
If $M_j \Fto \bf{0}$ then for almost every $r$ and for all sequences $p_j$
 we have (\ref{to-0}).
\end{lem}

\begin{thm}(c.f. \cite{Sormani-AA})\label{Flat-Arz-Asc}
Suppose $M_i=(X_i, d_i, T_i)$ are integral current spaces and
\be
M_i \Fto M_\infty,
\ee
and suppose we have Lipschitz maps into
a compact metric space $Z$, % CHANGED FROM W TO Z 
% to have notation be consistent
\be
F_i: X_i \to Z \textrm{ with } \Lip(F_i)\le K,
\ee
then a subsequence converges to a Lipschitz map
\be
F_\infty: X_\infty \to Z \textrm{ with }\Lip(F_\infty)\le K.
\ee
  More
specifically, there exists distance preserving maps 
of the subsequence, $\varphi_i: X_i \to Z$,
such that 
\be
d_F^Z(\varphi_{i\#} T_i , \varphi_\infty T_\infty)\to 0
\ee
and for any sequence $p_i\in X_i$ converging to $p\in X_\infty$
(i.e. $d_Z(\varphi_i(p_i), \varphi_\infty(p))\to 0$), we have
\be
\lim_{i\to\infty}F_i(p_i)=F_\infty(p_\infty).
\ee
\end{thm}

\begin{thm}(c.f. \cite{Sormani-AA})\label{B-W-BASIC}
Suppose $M^m_i=(X_i, d_i, T_i)$ are integral current spaces 
which converge in the intrinsic flat sense to a 
nonzero integral current space 
$M^m_\infty=(X_\infty, d_\infty, T_\infty)$.
Suppose there exists $r_0>0$ and a sequence
$p_i \in M_i$ such that for almost every $r\in (0, r_0)$ we have
integral current spaces, $S(p_i,r)$, for all $i\in \mathbb{N}$ and
\be 
\liminf_{i\to \infty} d_{\mathcal{F}}(S(p_i,r),{\bf{0}}) =h_0>0.
\ee 
Then there exists a subsequence, also denoted $M_i$, such that
$p_{i}$ converges to $p_\infty\in \bar{X}_\infty$.
\end{thm}

\begin{thm}(c.f. \cite{Sormani-AA}) \label{Arz-Asc-Unif-Local-Isom}
Let $M_i=(X_i, d_i, T_i)$
and $M'_i=(X'_i,d'_i,T_i)$  be integral current spaces with
\be
\mass(M_i)\le V_0 \textrm{ and }\mass(\partial M_i) \le A_0
\ee
such that 
\be
M_i \Fto M_\infty \textrm{ and } M'_i \Fto M'_\infty.
\ee

Fix $\delta>0$.
Let $F_i: M_i \to M'_i$ be continuous maps which are isometries
on balls of radius $\delta$:
\be \label{iso-sat} 
\forall x\in X_i, \,\, F_i: \bar{B}(x,\delta) \to \bar{B}(F_i(x),r)\textrm{ is an isometry}
\ee

Then, when $M_\infty\neq {\bf{0}}$, we have $M'_\infty \neq {\bf{0}}$ and
there is a subsequence, also denoted $F_i$, which
converges to a (surjective) local current preserving isometry
\be
F_\infty: \bar{X}_\infty \to \bar{X}'_\infty \textrm{ satisfying (\ref{iso-sat})}.
\ee
More
specifically, there exists distance preserving maps
of the subsequence $\varphi_i: X_i \to Z$,
$\varphi'_i: X'_i \to Z'$,
such that 
\be
d_F^Z(\varphi_{i\#} T_i , \varphi_\infty T_\infty)\to 0 \textrm{ and }
d_F^{Z'}(\varphi'_{i\#} T'_i , \varphi'_\infty T'_\infty)\to 0
\ee
and for any sequence $p_i\in X_i$ converging to $p\in X_\infty$:
\be
\lim_{i\to\infty} \varphi_i(p_i)=\varphi_\infty(p) \in Z
\ee
we have
\be \label{iso-infty}
\lim_{i\to\infty}\varphi_i'(F_i(p_i))=\varphi_\infty'(F_\infty(p_\infty)) \in Z'.
\ee
When $M_\infty={\bf{0}}$ and $F_i$ are surjective, we have $M'_\infty={\bf{0}}$.
\end{thm}

%%%%%%%%%%%%%%%%%%%
%% SEWING 
\section{Sewing Riemannian Manifolds with Positive Scalar Curvature}

The main technique we will introduce in this paper is the construction of three dimensional manifolds with positive scalar curvature through a process we call ``sewing'' which involved gluing a sequence of tunnels along a curve.  We
apply Lemma~\ref{tunnellemma} which constructs Gromov-Lawson Schoen-Yau tunnels.  The lemma is proven in the Appendix.

\subsection{Gluing Tunnels between Spheres}

We begin by gluing tunnels between arbitrary collections of pairs of spheres as in 
Figure~\ref{fig-glue-mnfld}.  

\begin{prop}\label{prop-glue}   
Given a complete Riemannian manifold, $M^3$, and $A_{0} \subset M^3$ a compact subset with an even number of points $p_{i} \in A_{0}$, $i = 1, \ldots, n$, with pairwise disjoint contractible balls $B(p_i,\delta)$ which have constant positive sectional curvature $K$, for some $\delta>0$, define $A_{\delta} = T_{\delta}(A_{0})$ and
%\be
%A_{\delta} = T_{\delta}(A_{0})
%\ee 
%and 
\be\label{prop-glue-defA'}
	A_{\delta}' = A_{\delta} \setminus \left( \bigcup_{i=1}^n B(p_i,\delta/2) \right) 	
		\disjointunion \bigcup_{i=1}^{n/2} U_i
\ee
where $U_i$ are the tunnels as in Lemma~\ref{tunnellemma} connecting $\partial B(p_{2j+1},\delta/2)$ to $\partial B(p_{2j+2},\delta/2)$ for $j=0,1,\ldots,n/2-1$. 
Then
given any $\epsilon>0$, shrinking $\delta$ further, if necessary, we may create a new complete Riemannian manifold, $N^3$, 
\be\label{E:prop-glue}
	N^3 = (M^3 \setminus A_{\delta}) \disjointunion A_{\delta}'
\ee
satisfying
\be\label{prop-glue-volA'}
(1-\epsilon)\vol(A_{\delta}) \le \vol(A_{\delta}')\le \vol(A_{\delta})(1+\epsilon)
\ee
and
\be\label{prop-glue-vol-space}
(1-\epsilon)\vol(M^3)\le \vol(N^3) \le \vol(M^3) (1+\epsilon).
\ee
%Do we need to add the diameter facts from the appendix here?
%Not really since we never use this proposition, it was meant to be
%iterated to get the next proposition but Jorge didn't prove it that way.

If, in addition, $M^3$ has non-negative or positive scalar curvature, then so does $N^3$.
In fact,
\be \label{inf-scal1}
\inf_{x\in M^3} \Scal_x \ge \min \left\{0, \inf_{x\in N^3} \Scal_x\right\}
\ee
If $\partial M^3 \neq \emptyset$, the balls avoid the boundary
and $\partial M^3$ is isometric to $\partial N^3$.
\end{prop}

\bd % def sewn manifold
We say that we have glued the manifold to itself with a tunnel between the collection of pairs of sphere  $\partial B(p_i,\delta)$ to $\partial B(p_{i+1},\delta)$
for $i=1$ to $n-1$.   See Figure~\ref{fig-glue-mnfld}.
\ed

\begin{figure}[htbp]
\begin{center}
\includegraphics[width=3in]{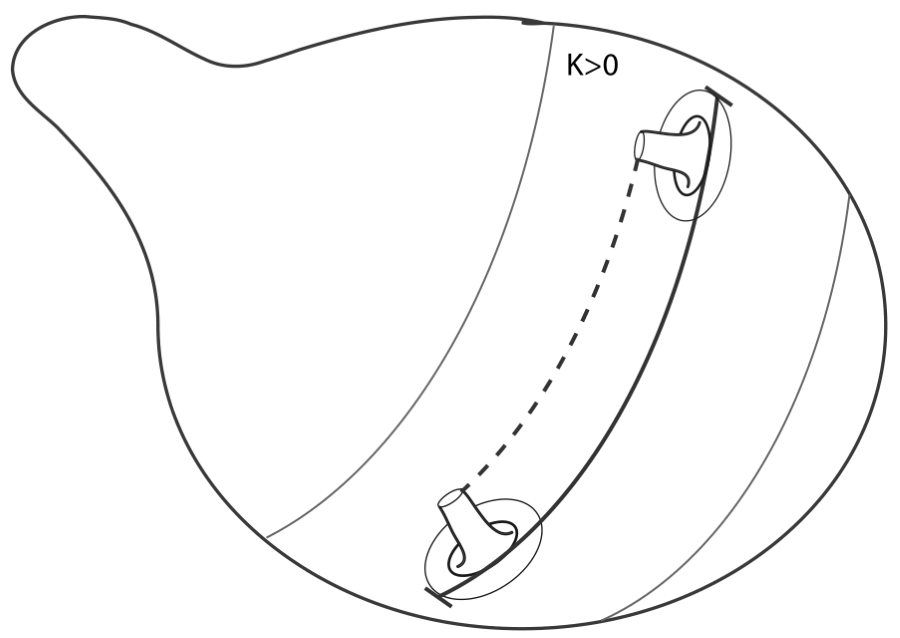}
\caption{Gluing two spheres with a tunnel.}
\label{fig-glue-mnfld}
\end{center}
\end{figure}

\begin{proof}
For simplicity of notation, set $A=A_{\delta}$ and $A'=A_{\delta}'$.

By induction on $n$ and Lemma~\ref{tunnellemma}, we see that $N^{3}$ can be given a metric of positive scalar curvature whenever $M^{3}$ has positive scalar curvature. 

Using the fact that the balls are pairwise disjoint and of the same volume, and (\ref{TL-volumeestU}) from Lemma~\ref{tunnellemma}, we have
the volume of $A'$ can be estimated: 
\begin{align*}
	\vol(A') &= \vol(A) - \sum_{i=1}^{n} \vol(B(p_i,\delta/2)) + \sum_{i=1}^{n/2} \vol(U_i) \\
		&= \vol(A) + \frac{n}{2} \cdot (\vol(U_i) - 2 \vol(B(p_i,\delta/2))) \\
	&\le \vol(A) + \frac{n}{2} \cdot (2\vol(B(p_i,\delta/2)) \cdot \epsilon)\\
	&= \vol(A) + \epsilon \cdot (n\vol(B(p_i,\delta/2))) \qquad \text{(by (\ref{TL-volumeestU}))}\\
	&\le \vol(A) + \epsilon \vol(A)
\end{align*}
which yields the right-hand side of (\ref{prop-glue-volA'}). 

Similarly,
\begin{align*}
	\vol(A') &= \vol(A) - \sum_{i=1}^{n} \vol(B(p_i,\delta/2)) + \sum_{i=1}^{n/2} \vol(U_i) \\
		&= \vol(A) + \frac{n}{2} \cdot (\vol(U_i) - 2 \vol(B(p_i,\delta/2))) \\
	&\ge \vol(A) + \frac{n}{2} \cdot (-2\vol(B(p_i,\delta/2)) \cdot \epsilon)\\
	&= \vol(A) - \epsilon \cdot (n\vol(B(p_i,\delta/2))) \qquad \text{(by (\ref{TL-volumeestU}))}\\
	&\ge \vol(A) - \epsilon \vol(A)
\end{align*}
which yields the left-hand side of (\ref{prop-glue-volA'}). 

% volume estimate for N
To estimate the volume of $N$ we will use the volume estimates for $A'$. Using (\ref{TL-volumeestU}) from Lemma~\ref{tunnellemma} again, we have
\begin{align*}
	\vol(N) &= \vol(M) - \vol(A) + \vol(A') \\
	&\le \vol(M) - \vol(A) + (1+\epsilon)\vol(A)\\
	&= \vol(M) + \epsilon \vol(A) \qquad \text{(by (\ref{TL-volumeestN}))}\\
	&\le \vol(M) + \epsilon \vol(M),
\end{align*}
which yields the right-hand side of (\ref{prop-glue-vol-space}). 

Similarly,
\begin{align*}
	\vol(N) &= \vol(M) - \vol(A) + \vol(A') \\
	&\ge \vol(M) - \vol(A) + (1-\epsilon)\vol(A)\\
	&= \vol(M) -\epsilon \vol(A)  \qquad \text{(by (\ref{TL-volumeestN}))}\\
	&\ge \vol(M) -\epsilon \vol(A),
\end{align*}
which yields the left-hand side of (\ref{prop-glue-vol-space}). 

Finally, observe that (\ref{inf-scal1}) follows since Lemma~\ref{tunnellemma} shows that the tunnels $U_{i}$ have positive scalar curvature. 
\end{proof} %ADDED 3-16-17

\subsection{Sewing along a Curve}

We now describe our process we call sewing along a curve, where a sequence of balls is taken to be located along curve much like holes created when stitching a thread.  We glue a sequence of tunnels to the boundaries of these balls as in Figure~\ref{fig-sewn-mnfld}.  We say that we have sewn the manifold along the curve $C$ through the given balls.  By gluing tunnels in this precise way we
are able to shrink the diameter of the edited tubular neighborhood around the
curve because travel along the curve can be conducted efficiently through the tunnels.

\begin{figure}[htbp]
\begin{center}
\includegraphics[width=3in]{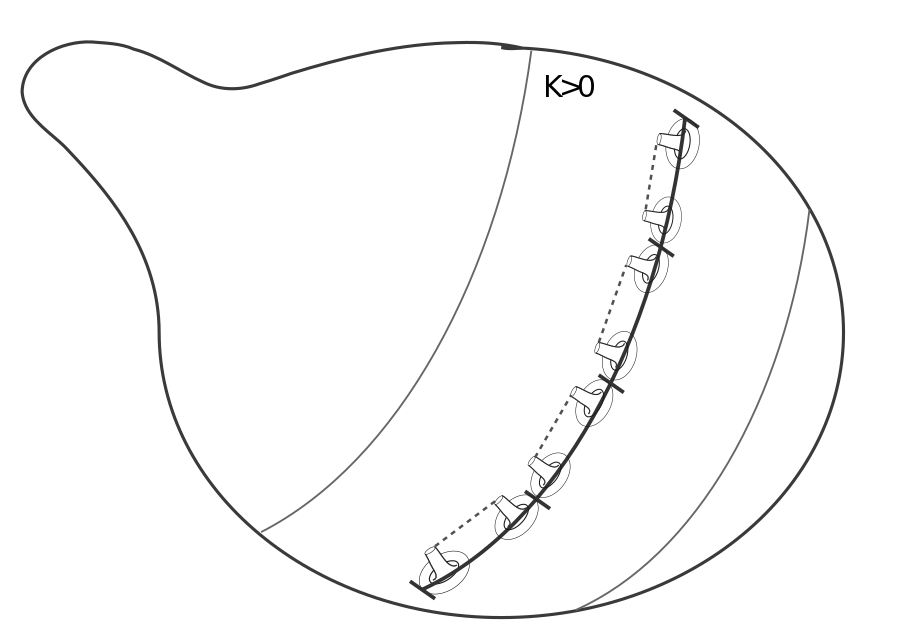}
\caption{Sewing a manifold through eight balls along a curve.}
\label{fig-sewn-mnfld}
\end{center}
\end{figure}

\begin{prop}\label{sewn-curve}
Given a complete Riemannian manifold, $M^3$, and $A_{0}\subset M^3$ Riemannian isometric to an embedded curve, $C:[0,1]\to \mathbb{S}^3_K$ 
possibly with $C(0)=C(1)$ and parametrized proportional to arclength, in a 
standard sphere of constant sectional curvature $K$,
define $A_{a} = T_{a}(A_{0})$ as in Proposition~\ref{prop-glue} and assume that $A_{a}$ is Riemannian isometric to $T_{a}(C) \subset \mathbb{S}^3_K$. 
Then, given any $\epsilon>0$ there exists $n$ sufficiently large and $\delta=\delta(\epsilon,n,C,K)>0$ sufficiently small as in (\ref{sewn-curve-choiceofdelta}) so that
 we can ``sew along the curve'' to create a new complete Riemannian 
 manifold $N^3$,
\be
N^3 = (M^3 \setminus A_{\delta} )\disjointunion A_{\delta}',
\ee
exactly as in Proposition~\ref{prop-glue}, for 
\be\label{sewn-curve-choiceofdelta}
\delta=\delta(\epsilon,n,C,K) \textrm{  such that  } 
	\delta<a,\,
	\lim_{n\to \infty} n \cdot h(\delta)= 0,
	\textrm{ and }
	\lim_{n\to \infty} n \cdot \delta = 0,
\ee
%[Basilio: I added $\delta_n<a$ as a fail-safe]
where $h$ is defined in Lemma~\ref{tunnellemma}
and the disjoint balls $B(p_i,\delta)$ are to be centered at 
\be
p_{2j+1}=C\left(\frac{j}{n}+\frac{\delta}{L(C)}\right) \qquad 
p_{2j+2}=C\left(\frac{j+1}{n} -\frac{\delta}{L(C)}\right) \qquad j=0,1,\ldots,n-1
\ee
and 
\be
A_{\delta}'= A_{\delta} \setminus \left(\bigcup_{i=1}^{2n} B(p_i,\delta/2)\right)\disjointunion \bigcup_{j=0}^{n-1} U_{2j+1}.  
\ee
Thus, the tunnels $U_{2j+1}$ connect $\partial B(p_{2j+1},\delta)$ to $\partial B(p_{2j+2},\delta)$ for $j=0,1,\ldots, n-1$. 

Furthermore,
\be\label{sewn-curve-tubular}
(1-\epsilon)\vol(A_{\delta}) \le \vol(A_{\delta}')\le \vol(A_{\delta})(1+\epsilon) 
\ee
and
\be\label{sewn-curve-vol}
(1-\epsilon) \vol(M^3)\le \vol(N^3) \le \vol(M^3) (1+\epsilon)
\ee
and
\be \label{sewn-curve-diam}
\diam(A_{\delta}')\le H(\delta)= L(C)/n + (n+1)\, h(\delta)+(5n+2)\, \delta.
\ee

Since 
\be\label{sewn-curve-Hto0}
\lim_{\delta\to 0} H(\delta)=0 \textrm{ uniformly for } K\in (0,1], 
\ee
we say we have sewn the curve, $A_{0}$, arbitrarily short. 

If, in addition, $M^3$ has non-negative or positive scalar curvature, then so does $N^3$.
In fact,
\be \label{inf-scal-1}
\inf_{x\in M^3} \Scal_x \ge \min \left\{0, \inf_{x\in N^3} \Scal_x\right\}
\ee
If $\partial M^3 \neq \emptyset$, the balls avoid the boundary
and $\partial M^3$ is isometric to $\partial N^3$.   
\end{prop}

\begin{proof}
By the fact that $C$ is embedded, for $n$ sufficiently large, the balls in the statement are disjoint even when $C(0)=C(1)$ so we may apply Propositon~\ref{prop-glue} to get (\ref{sewn-curve-tubular}) and (\ref{sewn-curve-vol}).  

For simplicity of notation, let $A=A_{\delta}$ and $A'=A_{\delta}'$. 

We now verify the diameter estimate of $A'$, (\ref{sewn-curve-diam}). To do this we define sets $C_i \subset A'$ which correspond to the sets $\partial B(p_i,\delta/2) \subset A$ which are unchanged because they are the boundaries of the edited regions:
\be
	C_i \cup C_{i+1} = \partial U_{i},
\ee
whenever $i$ is an odd value. Let 
\be
	U=\bigcup_{j=0}^{n-1} U_{2j+1}.
\ee

Let $x$ and $y$ be arbitrary points in $A'$. We claim that there exists $j,k \in \{1,\ldots,2n\}$ such that
\be\label{claim:gettotunnel}
d_{A'}(x,C_j) < \delta + L(C)/(2n)+h(\delta) \textrm{ and } 
d_{A'}(y,C_k) < \delta + L(C)/(2n)+h(\delta)
\ee

By symmetry we need only prove this for $x$. Note that in case I where
\be
x \in A' \setminus U = A \setminus \bigcup_{i=1}^{2n} B(p_i,\delta/2)
\ee
we can view $x$ as a point in $A$. Let $\gamma_1 \subset A$ be the shortest path from $x$ to the closest point $c_x \in C[0,1]$ so that   $L(\gamma_1) < \delta$.  % 

If 
\be\label{claimgamma1touchesball}
\gamma_1 \cap B(p_j,\delta/2) \ne \emptyset
\ee
then 
\be
d_{A'\setminus U}(x,C_j)<\delta
\ee
and we have that (\ref{claim:gettotunnel}) holds. Otherwise, still in Case I, if (\ref{claimgamma1touchesball}) fails then we have 
\begin{eqnarray}
d_{A'\setminus U}(x,C_j) &\le& d_{A'\setminus U}(x,c_x) + d(c_x,C_j) 
	\qquad \textrm{(by the triangle inequality)}\\
	&<& \delta+\frac{L(C)}{2n},
\end{eqnarray}
where the last inequality follows from $d_{A'\setminus U}(x,c_x) \le L(\gamma_1) <\delta$ and the fact that $c_x \in C([0,1])$ is at most $L(C)/(2n)$ away from the boundary of the nearest tunnel. 

Alternatively, we have case II where $x \in U$. In this case, there exists $j$ such that $x \in U_{2j+1}$ and so
\be
d_{A'}(x,C_{2j+1}) \le \diam(U_{2j+1}) \le h(\delta).
\ee
Thus, we have the claim in (\ref{claim:gettotunnel}). 

We now proceed to prove (\ref{sewn-curve-diam}) by estimating $d_{A'}(x,y)$ for $x,y \in A'$. If $j=k$ in (\ref{claim:gettotunnel}), then $d_{A'}(x,y) \le 2(\delta + L(C)/(2n)+h(\delta))$ and we are done. Otherwise, by (\ref{claim:gettotunnel}) and the triangle inequality, we have 
\begin{eqnarray}
d_{A'}(x,y) &\le& d_{A'}(x,C_j) + d_{A'}(y,C_k) 
	+ \sup \{ d_{A'}(z,w) \mid z \in C_j, w\in C_k \}  \\
	&\le& 2(\delta + L(C)/(2n)+h(\delta)) + \sup \{ d_{A'}(z,w) \mid 
		z \in C_j, w\in C_k \}.\label{diam-needsup}
\end{eqnarray}
Without loss of generality, we may assume that $j<k$ and that $j$ is odd. Thus, $C_j \subset \partial U_j$. If $k$ is also odd then by the triangle inequality 
\begin{eqnarray}\label{kodd}
\sup \{ d_{A'}(z,w) \mid z \in C_j, w\in C_k \} &\le&
	\diam(U_j) + \dist(U_j,U_{j+2}) \\
	&& + \diam(U_{j+2}) +\cdots +\diam(U_{k-2}) \notag\\ 
	&& + \dist(U_{k-2},U_k) \notag
\end{eqnarray}
and, when $k$ is even,
\begin{eqnarray}\label{keven}
\sup \{ d_{A'}(z,w) \mid z \in C_j, w\in C_k \} &\le&
	\diam(U_j) + \dist(U_j,U_{j+2}) \\
	&& + \diam(U_{j+2}) +\cdots +\diam(U_{k-2}) \notag\\ 
	&& + \dist(U_{k-2},U_{k-1}) + \diam(U_{k-1}). \notag
\end{eqnarray}

We know that $\diam(U_j) = \cdots =\diam(U_k) \le h(\delta)$ from (\ref{TL-diameterU}) of Lemma~\ref{tunnellemma}, and that the distance between any two adjacent tunnels is the same, and that there are at most $n$ tunnels. Thus, in either case (\ref{kodd}) or (\ref{keven}) we have 
\be\label{diam-sup1}
\sup \{ d_{A'}(z,w) \mid z \in C_j, w\in C_k \}
	\le n\, h(\delta) + n\cdot \dist(U_j,U_{j+2}).
\ee
and by construction the distance between adjacent tunnels is 
\begin{eqnarray}
\dist(U_j,U_{j+2}) &\le& \diam(C_{j+1}) + \dist(C_{j+1},C_{j+2}) + \diam(C_{j+2}) \\
	&\le& \pi (\delta/2) + \delta + \pi (\delta/2) < 5\delta\label{diam-sup2}
\end{eqnarray}
since the balls $B(p_i,\delta/2)$ have constant sectional curvature $K$. 

Therefore, combining (\ref{diam-needsup}), (\ref{diam-sup1})  and (\ref{diam-sup2}) we conclude that 
\be
d_{A'}(x,y) \le 2(\delta + L(C)/(2n)+h(\delta))+ n\, h(\delta) + 5n \delta
\ee
which is the desired diameter estimate (\ref{sewn-curve-diam}).

We observe that by our choice of $\delta$ satisfying (\ref{sewn-curve-choiceofdelta}) and the fact that $h(\delta) = O(\delta)$ from Lemma~\ref{tunnellemma} we have that (\ref{sewn-curve-Hto0}) holds. 

Finally, observe that (\ref{inf-scal-1}) follows since Lemma~\ref{tunnellemma} shows that the tunnels $U_{i}$ have positive scalar curvature. 
\end{proof}

%%%%%%%%%%%%%%%%%%%%%%%%%%%%%%%%%%
% Pulled Metric Spaces
\section{Pulled String Spaces}\label{sect-pulled-string}

The following notion of a pulled string metric space captures the idea that if a metric space is a patch of cloth and a curve in the patch is sewn with a string, then one can pull the string tight, identifying the entire curve as a single point, thus creating a new metric space.  This notion was first described to the third author by
Burago when they were working ideas related to \cite{Burago-Ivanov-Area}. 
See Figure~\ref{fig-pulled-string}.

\begin{prop} \label{pulled-string}
The notion of a metric space with a pulled string is 
a metric space $(Y, d_Y)$ constructed from a metric space $(X,d_X)$ 
with a curve $C:[0,1]\to X$, so that
\be\label{pulled-string-def1}
Y = X \setminus C[0,1] \disjointunion \{p_0\}, \qquad p_0=C(0),
\ee
where for $x_i \in Y$ we have
\be\label{pulled-string-def2}
	d_Y(x, p_0) = \min \{ d_X(x, C(t)) : \, t\in [0,1]\}
\ee
and for $x_i \in X \setminus C[0,1]$  we have
\be\label{pulled-string-def3}
d_Y(x_1, x_2) =\min\left\{\, d_X(x_1, x_2), \min\{d_X(x_1, C(t_1)) + d_X(x_2, C(t_2)): \, t_i \in [0,1] \}\, \right\}. 
\ee
If $(X,d,T)$ is a Riemannian manifold then $(Y,d,\psi_\#T)$
is an integral current space whose mass measure is the 
Hausdorff measure on $Y$ and
\be\label{pulledcurveHm}
\mathcal{H}_Y^m(Y)=\mathcal{H}_X^m(X)-\mathcal{H}_X^m(K).
\ee
If $(X, d_X, T)$ is an integral current space then $(Y, d_Y, \psi_{\#}T)$ is
also an integral current space where $\psi: X\to Y$ such that $\psi(x)=x$ for all $x\in X\setminus C[0,1]$ and $\psi(C(t))=p_0$ for all $t\in [0,1]$.   So that
\be\label{massofpulledcurve}
\mass(\psi_{\#}T)=\mass(T)    
\ee
\end{prop}

\begin{figure}[htbp]
\begin{center}
\includegraphics[scale=0.3]{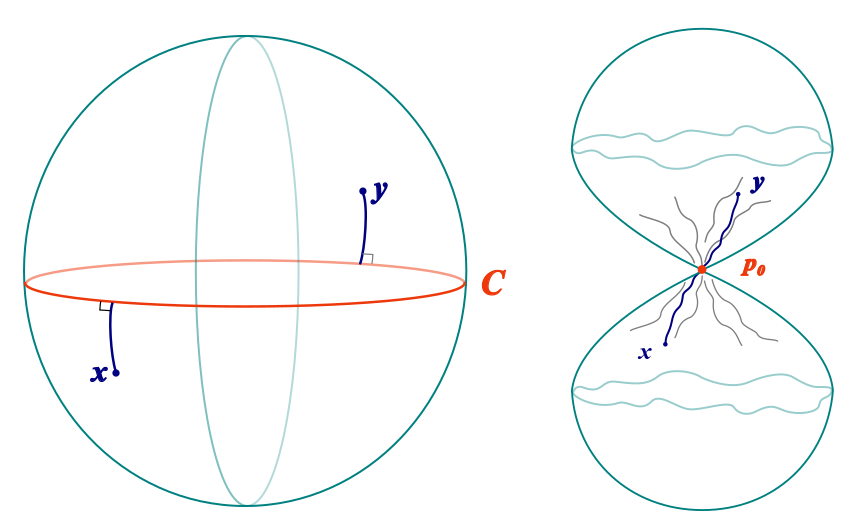}
\caption{A two sphere with the equator pulled to a point.}
\label{fig-pulled-string}
\end{center}
\end{figure}

We will in fact prove this proposition as a consequence of
two lemmas about spaces with arbitrary compact subsets pulled to
a point.  Lemma~\ref{pulled-subset-1} proves such a space is
a metric space and Lemma~\ref{pulled-subset-2} proves
(\ref{pulledcurveHm}) and (\ref{massofpulledcurve}).

\subsection{Pulled string spaces are metric spaces}

% pulled subset
\begin{lem} \label{pulled-subset-1}
Given a metric space $(X, d_X)$ and a compact set $K \subset X$
we may define a new metric space $(Y, d_Y)$ by pulling the set $K$ to a point $p_0 \in K$ by setting 
\be\label{pulled-set-def1}
Y := X \setminus K \disjointunion \{p_0\}, \qquad p_0 \in K \, \textrm{fixed}, 
\ee
and, for $x \in Y$, we have 
\be\label{pulled-set-def2}
d_Y(x, p_0) = \min\{ d_X(x, y) : \, y\in K\}
\ee
and, for $x_i \in Y \setminus \{p_0\}$, we have
\be\label{pulled-set-def3}
d_Y(x_1, x_2) =\min\left\{ d_X(x_1, x_2), \min\{d_X(x_1, y_1) + d_X(x_2, y_2): \, y_i\in K\} \right\}. 
\ee
\end{lem}

%\br
%We note that the result concerning Hausdorff measures, (\ref{pulledsetHm}), holds for any $s \in (0,\infty)$ and not just for integer $s=m$ as stated. 
%\er

\begin{proof}
% proof that it is a metric space
We first prove that $(Y,d_Y)$ is a metric space. By definition, it is easy to see that $d_Y$ is non-negative and symmetric. To prove that $d_Y$ satisfies the axiom of positivity, assume $x_1=x_2$. Then either $x_i=p_0$, and $d_Y(x_1,x_2)=0$ by definitions (\ref{pulled-set-def1})--(\ref{pulled-set-def2}), or $x_i \ne p_0$ and $d_X(x_1,x_2)=0$ so by (\ref{pulled-set-def3}) we have $d_Y(x_1,x_2) \le d_X(x_1,x_2)=0$. 
Conversely, if $d_Y(x_1,x_2)=0$ then either $d_X(x_1,x_2)=0$ or 
\be\label{pulled-set-4}
	0= \min \{ d_X(x_1,y_1)+d_X(x_2,y_2) \mid y_i \in K \}.
\ee
In the first case, $x_1=x_2$ since $d_X$ is a metric, so assume otherwise. Then $d_X(x_1,x_2) \ne 0$ and (\ref{pulled-set-4}) holds. Being that (\ref{pulled-set-4}) is a sum of non-negative numbers, it follows that $d_X(x_1,y_1)=0$ and $d_X(x_2,y_2)=0$ for some $y_i \in K$. Hence, $x_i=y_i$ which is impossible by the definition of $Y$ unless $x_1=x_2=p_0$ which yields a contradiction. This proves that $d_Y$ satisfies positivity.  

% two observations

Next, let us note that by virtue of (\ref{pulled-set-def2}) and (\ref{pulled-set-def3}), we always have
\be\label{dYlessdX}
	d_Y(x_1,x_2) \le d_X(x_1,x_2), \qquad \forall\, x_1,x_2 \in Y
\ee
and
\be\label{dYsum}
\textrm{if }d_Y(x_1,x_2) \ne d_X(x_1,x_2) \implies 
d_Y(x_1,x_2) = d_X(x_1,y_1)+d_X(x_2,y_2).
\ee
for some $y_i \in K$. 

% Triangle Inequality

We now verify the triangle inequality: for any $x_1, x_2, x_3 \in Y$, we need to prove
\be\label{triangleineq}
d_Y(x_1,x_2) \le d_Y(x_1,x_3) + d_Y(x_3,x_2).	
\ee 

It will be convenient to define $y_i \in K$ such that
\be\label{defofyi}
	d_X(x_i,y_i) = \min \{ d_X(x_i,y) \mid y \in K\} 
		\textrm{  for  }i=1,2,3.
\ee

%case I
Assume in Case I that $d_Y(x_1,x_2) \ne d_X(x_1,x_2)$. Then by (\ref{dYsum}) and (\ref{defofyi}),  
\be\label{caseIsum}
d_Y(x_1,x_2) = d_X(x_1,y_1)+d_X(x_2,y_2).
\ee

We have three possibilities: (i) $d_Y(x_1,x_3) \ne d_X(x_1,x_3)$ and $d_Y(x_2,x_3) \ne d_X(x_2,x_3)$; (ii) $d_Y(x_1,x_3) = d_X(x_1,x_3)$ and $d_Y(x_2,x_3) = d_X(x_2,x_3)$; and (iii) (without loss of generality) $d_Y(x_1,x_3) \ne d_X(x_1,x_3)$ and $d_Y(x_2,x_3) = d_Y(x_2,x_3)$.

In Case I (i), we have  
\begin{eqnarray*}
d_Y(x_1,x_2) &=& d_X(x_1,y_1) + d_X(x_2,y_2) 
	\qquad \textrm{(by (\ref{caseIsum}))}\\
&\le& d_X(x_1,y_1) +d_X(x_3,y_3)+ d_X(x_2,y_2) + d_X(x_3,y_3)\\
&=& d_Y(x_1,x_3) + d_Y(x_2,x_3). 
	\qquad \textrm{(by assumption (i), (\ref{dYsum}), and (\ref{defofyi}))}
\end{eqnarray*}

In Case I (ii), we have
\begin{eqnarray*}
d_Y(x_1,x_2) &\le& d_X(x_1,x_2) \qquad \textrm{(by (\ref{dYlessdX}))}\\
&\le& d_X(x_1,x_3) + d_X(x_2,x_3) \\%\qquad \textrm{(by triangle ineq. for $d_X$)}\\
&=& d_Y(x_1,x_3) + d_Y(x_2,x_3). \qquad \textrm{(by assumption (ii))}
\end{eqnarray*}

In Case I (iii), we have
\begin{eqnarray}
d_X(x_2,y_2) &=& \min \{ d_X(x_2,K) \mid y \in K\} 
	\qquad \textrm{(by (\ref{defofyi}))} \notag\\
&\le& d_X(x_2,y_3) \notag\\
&\le& d_X(x_2,x_3) + d_X(x_3,y_3) \\
	%\qquad \textrm{(by triangle ineq. for $d_X$)}\notag\\
&\le& d_Y(x_2,x_3) + d_X(x_3,y_3) 
	\qquad \textrm{(by assumption (iii))}\label{caseI-iii}
\end{eqnarray}
so that
\begin{eqnarray}
d_Y(x_1,x_2) &=& d_X(x_1,y_1) + d_X(x_2,y_2) 
	\qquad \textrm{(by (\ref{caseIsum}))} \notag\\
&\le& d_X(x_1,y_1)+ d_Y(x_2,x_3) + d_X(x_3,y_3) 
	\qquad \textrm{(by (\ref{caseI-iii}))}\notag\\
&=& d_Y(x_1,x_3)+d_Y(x_2,x_3). \qquad \textrm{(by assumption (iii))}\notag
\end{eqnarray}
This proves the triangle inequality, (\ref{triangleineq}), in Case I. 
% case II
Next, we assume, in Case II, that $d_Y(x_1,x_2) = d_X(x_1,x_2)$.

Again, we have three possibilities: (i) $d_Y(x_1,x_3) \ne d_X(x_1,x_3)$ and $d_Y(x_2,x_3) \ne d_X(x_2,x_3)$; (ii) $d_Y(x_1,x_3) = d_X(x_1,x_3)$ and $d_Y(x_2,x_3) = d_X(x_2,x_3)$; and (iii) (without loss of generality) $d_Y(x_1,x_3) \ne d_X(x_1,x_3)$ and $d_Y(x_2,x_3) = d_Y(x_2,x_3)$.

In Case II (i), we have
\begin{align*}
	d_Y(x_1,x_2) &= d_X(x_1,x_2) \\
		&\le  d_X(x_1,y_1)+d_X(x_2,y_2) \qquad \textrm{(by (\ref{caseIsum}))} \notag\\		&\le d_X(x_1,y_1)+d_X(x_3,y_3)
			+d_X(x_2,y_2)+d_X(x_3,y_3) \\
		&= d_Y(x_1,x_3) + d_Y(x_2,x_3). \qquad
			\textrm{(by assumption (i), (\ref{dYsum}), and (\ref{defofyi}))}
\end{align*}

In Case II (ii), (\ref{triangleineq}) follows immediately from the triangle inequality for $d_X$. 

Finally, in Case II (iii), 
\begin{align*}
d_Y(x_1,x_2) &= d_X(x_1,x_2) \\
	&\le  d_X(x_1,y_1)+d_X(x_2,y_3) \qquad \textrm{(by (\ref{caseIsum}))} \notag\\
	&\le d_X(x_1,y_1)+d_X(x_2,x_3) + d_X(x_3,y_3) \\%\qquad trian\\
	&= d_Y(x_1,x_3) + d_Y(x_2,x_3), \qquad
			\textrm{(by assumption (iii), (\ref{dYsum}), and (\ref{defofyi}))}
\end{align*}
which completes the proof.
\end{proof}

\subsection{Hausdorff Measures and Masses of Pulled String Spaces}

\begin{lem} \label{pulled-subset-2}
If $(X, d_X, T)$ is an integral current space with a compact
subset $K \subset X$ then $(Y, d_Y, \psi_{\#}T)$ is
also an integral current space
where $(Y, d_Y)$ is defined as in Lemma~\ref{pulled-subset-1} and
where $\psi: X\to Y$ such that $\psi(x)=x$ for all $x\in X\setminus K$
and $\psi(q)=p_0$ for all $q\in K$.   In addition
\be\label{massofpulledset}
\mass(\psi_{\#}T)=\mass(T) - ||T||(K)   
\ee
If $(X,d_X,T)$ is a Riemannian manifold then $(Y,d_Y,\psi_\#T)$
is an integral current space whose mass measure is the 
Hausdorff measure on $Y$ and
\be\label{pulledsetHm}
\mathcal{H}_Y^m(Y)=\mathcal{H}_X^m(X)-\mathcal{H}_X^m(K).
\ee
\end{lem}

\begin{proof}
Next, suppose that $(X,d_X,T)$ is an $m$-dimensional integral current space. We must show that  $(Y,d_Y,\psi_\# T)$ is an integral current space. We first observe that $\psi$ as defined in the statement of the proposition is a 1-Lipschitz function: for $x,y \in X\setminus K$, there is no ambiguity so we may view them as elements of $Y\setminus\{p_0\}$ and $d_Y(\psi(x),\psi(y))=d_Y(x,y) \le d_X(x,y)$ by definition of $d_Y$. Otherwise, we may assume, without loss of generality, that $x \in K$ and $y \notin K$. In this case, $d_Y(\psi(x),\psi(y))=d_Y(p_0,\psi(y))=d_Y(p_0,y)=\min\{d_X(z,y) : z \in K\} \le d_X(x,y)$, as $x\in K$.   Thus, $\psi_\# T$ is an integral current on $Y$ since $\psi$ is a 1-Lipschitz function and the well-known inequality 
\be
	\|\psi_\# T\| \le \Lip(\psi)^m \|T\|
\ee
implies that $\psi_\# T$ has finite mass because $T$ does. To show that $(Y,d_Y,\psi_\# T)$ is an integral current space there remains to show that it is completely settled, or $\psi_\# T$ has positive density at $p_0$. 

Let $f:Y \to \R$ be a bounded Lipschitz map and $\pi_j:Y \to \R$ be Lipschitz maps. Then
\begin{align*}
(\psi_\# T)(f,\pi_1,\ldots,\pi_m) 
&= T(f \circ \psi, \pi_1 \circ \psi, \ldots, \pi_m \circ \psi) \\
&= T(f\cdot \One_{X\setminus K} + f(p_0)\cdot \One_K, \pi_1 \circ \psi, \ldots, \pi_m \circ \psi)\\
&= T(f\cdot \One_{X\setminus K}, \pi_1 \circ \psi, \ldots, \pi_m \circ \psi)
	+ f(p_0)T(\One_K, \pi_1 \circ \psi, \ldots, \pi_m \circ \psi)\\
&= T(f\cdot \One_{X\setminus K}, \pi_1 \circ \psi, \ldots, \pi_m \circ \psi) + 0
\end{align*}
by locality since $\pi_i \circ \psi$ are constant on $\{\One_K \ne 0\}$ (see \cite{AK}) so
\begin{align*}
(\psi_\# T)(f,\pi_1,\ldots,\pi_m) 
&= T(f\cdot \One_{X\setminus K}, \pi_1 \circ \psi, \ldots, \pi_m \circ \psi)\\
&= (T \rstr \One_{X\setminus K})(f, \pi_1 \circ \psi, \ldots, \pi_m \circ \psi) \\
&= (T \rstr \One_{X\setminus K})(f\circ\psi, \pi_1 \circ \psi, \ldots, \pi_m \circ \psi)\\
&\qquad	\qquad \text{because $\psi(x)=x$ on $X\setminus K$,}\\
&= \psi_\#(T \rstr \One_{X\setminus K})(f,\pi_1,\ldots,\pi_m).
\end{align*}
So,  using the characterization of mass from \cite{AK}, (2.6) of Proposition 2.7, 
\begin{align*}
\mass(\psi_\# T)&=\mass(\psi_\#(T \rstr \One_{X\setminus K}))\\
&=\mass(T\rstr  \One_{X\setminus K})
\end{align*}
because $\psi(x)=x$ on $X\setminus K$, so since $\mass(\cdot)=\|\cdot\|$,
\begin{align*}
(\psi_\# T)(f,\pi_1,\ldots,\pi_m) 
&=\|T\rstr  \One_{X\setminus K}\|(X)\\
&=\sup \left\{ \sum_{j=1}^\infty |(T\rstr  \One_{X\setminus K})(\One_{A_j},\pi_1^j,\ldots,\pi_m^j)| \right\},
\end{align*}
where the supremum is taken over all Borel partitions $\{A_{j}\}$ of $X$ such that $X = \cup_j A_j$ and all Lipschitz functions $\pi_i^j \in \Lip(X)$ with $\Lip(\pi_i^j) \le 1$, then continuing
\begin{align*}
(\psi_\# T)(f,\pi_1,\ldots,\pi_m) 
&=\sup \left\{ \sum_{j=1}^\infty |T(\One_{X\setminus K} \cdot \One_{A_j},\pi_1^j,\ldots,\pi_m^j)| \right\}\\
%&\qquad \qquad X = \cup_j A_j \text{ Borel partition, } \pi_i^j \in \Lip(X), \Lip(\pi_i^j) \le 1 \}\\
&=\sup \left\{ \sum_{j=1}^\infty |T(\One_{\tilde{A}_j},\tilde{\pi}_1^j,\ldots,\tilde{\pi}_m^j)| \right\}, 
\end{align*}
where the second supremum is taken over all Borel partitions $\{\tilde{A}_{j}\}$ of $X\setminus K$ such that $X\setminus K = \cup_j \tilde{A}_j$ and all Lipschitz functions $\tilde{\pi}_i^j \in \Lip(X\setminus K)$ with $\Lip(\tilde{\pi}_i^j) \le 1$.
So, by the characterization of mass we have 
\begin{align*}
(\psi_\# T)(f,\pi_1,\ldots,\pi_m) 
&=\sup \left\{ \sum_{j=1}^\infty |T(\One_{\tilde{A}_j},\tilde{\pi}_1^j,\ldots,\tilde{\pi}_m^j)| \right\}\\
&=||T||(X\setminus K)\\
&= ||T||(X) - ||T|(K) \\
&=\mass(T)-||T||(K),
\end{align*}
which proves (\ref{massofpulledset}).

% Hausdorff measure on Y

Finally, assume that the $m$-dimensional integral current space $(X,d_X,T)$ is a Riemannian manifold. We show that the mass measure of  $(Y,d_Y,\psi_\#T)$ is the Hausdorff measure on $(Y,d_Y)$.  

We claim that 
\be\label{rstrHausd}
	\mathcal{H}_Y^m \rstr (Y\setminus \{p_0\}) = \mathcal{H}_X^m \rstr (X\setminus K).
\ee

First, observe that since $\psi$ is 1-Lipschitz,
\[
\mathcal{H}_Y^m (\psi(X\setminus K)) \le (\Lip(\psi))^{m} \mathcal{H}_X^m(X\setminus K), 
\]
by Proposition 3.1.4 on page 37 from \cite{Ambrosio-Tilli-AoMS}, hence
\[
\mathcal{H}_Y^m (Y\setminus \{p_0\}) \le \mathcal{H}_X^m(X\setminus K). 
\]
Thus, there remains to show the opposite inequality in (\ref{rstrHausd}).

Define sets
\[
C_{j} = \{ y \in Y \mid d_{Y}(y,p_{0}) \ge 1/j \}
\]
for each $j \in \N$. Then the $C_{j}$ are closed sets, $C_{j} \subset C_{j+1}$ and $Y\setminus \{p_{0}\} = \cup_{j \in \N} C_{j}$. So we may use Theorem 1.1.18 from \cite{Ambrosio-Tilli-AoMS}:
\be\label{hausd-limit1}
\mathcal{H}_Y^m (Y\setminus \{p_0\}) = \mathcal{H}_Y^m ( \cup_{j \in \N} C_{j}) =
\lim_{j \to \infty} \mathcal{H}_Y^m (C_{j}).
\ee

Consider, for each $j \in N$,  
\[
D_{j} = \psi^{-1}(C_{j}) = \{ x \in X \mid d_{X}(x,K) \ge 1/j \}
\]
which are closed in $X$, $D_{j} \subset D_{j+1}$, and $X \setminus K = \cup_{j \in \N} D_{j}$. Using Theorem 1.1.8 from \cite{Ambrosio-Tilli-AoMS} again: 
\be\label{hausd-limit2} 
\mathcal{H}_X^m (X\setminus K) = \mathcal{H}_X^m ( \cup_{j \in \N} D_{j}) =
\lim_{j \to \infty} \mathcal{H}_X^m (D_{j}).
\ee

Next, we claim that
\be\label{hausd-DCineq}
\mathcal{H}_X^m (D_{j}) \le \mathcal{H}_Y^m(C_{j}), \qquad j \in \N.
\ee
Fix $j$. Fix $\delta < \frac{1}{2j}$. Let $\{ E_{l} \}_{l \in \N}$ be a countable cover of $C_{j}$ with $\diam(E_{l})<\delta$, for all $l$. Then 
\be\label{hausd-distp0El}
\dist(E_{l},p_{0})>\frac{1}{2j}, \qquad l \in \N.
\ee
To see this, assume otherwise. Then since $\dist_{Y}(p_{0},E_{l})<\frac{1}{2j}$ and the definition of distance (as an infimum), there is $e \in E_{l}$ such that $d_{Y}(p_{0},e)<\frac{1}{2j}$. Now, we also know that $E_{l} \cap C_{j} \ne \emptyset$. So, there is $c \in C_{j} \cap E_{l}$. So, $d_{Y}(e,c) \le \diam_{Y}(E_{l}) < \delta < \frac{1}{2j}$. Also, by the triangle inequality, $d_{Y}(p_{0},c) \le d_{Y}(p_{0},e)+d_{Y}(e,c) < 1/j$. But this contradicts that $c \in C_{j}$ as by definition of $C_{j}$, $d_{Y}(p_{0},c)>1/j$. 

Next, we show that 
\be\label{hausd-diamequal}
\diam_{Y}(E_{l}) = \diam_{X}(\psi^{-1}(E_{l})),
\ee
i.e. $\psi^{-1}$ is an isometry when restricted to $\{E_{l}\}$. In fact, we prove 
\[
d_{X}(\psi^{-1}(a),\psi^{-1}(b)) = d_{Y}(a,b), \qquad \forall\, a,b \in E_{l}, j \in \N.
\]

Let $a,b \in E_{l}$. Then since $\diam(E_{l})<\delta<\frac{1}{2j}$ we have $d_{Y}(a,b) \le \diam_{Y}(E_{l}) < \delta < \frac{1}{2j}$, so
\be\label{hausd-abinEl1}
d_{Y}(a,b) < \frac{1}{2j}.
\ee
By definition of the distance $d_{Y}$, since $\psi^{-1}(a)=a$ and $\psi^{-1}(b)=b$,
\[
d_{Y}(a,b) = \min \big \{ d_{X}(a,b),\, \min \{\,  d_{X}(a,k_{1})+d_{X}(b,k_{2}) \mid k_{i} \in K\}\, \big \}.
\]
If $d_{Y}(a,b)=d_{X}(a,b)$, we're done. If not, then there exists $k_{1},k_{2} \in K$ so that
\be\label{hausd-abinEl2}
d_{Y}(a,b)=d_{X}(a,k_{1})+d_{X}(b,k_{2}).
\ee
By (\ref{hausd-distp0El}), 
\[
d_{Y}(a,p_{0}) \ge \frac{1}{2j} \qquad \text{and} \qquad d_{Y}(b,p_{0}) \ge \frac{1}{2j}
\] 
which implies
\[
\dist_{X}(a,K) \ge \frac{1}{2j} \qquad \text{and} \qquad \dist_{X}(b,K) \ge \frac{1}{2j}. 
\]
But then
\begin{align*}
\frac{1}{j} &\le \dist_{X}(a,K) + \dist_{X}(b,K) \\
	&\le d_{X}(a,k_{1})+d_{X}(b,k_{2}) \\
	&= d_{X}(a,b) \qquad \text{(by (\ref{hausd-abinEl2}))} \\
	&< \frac{1}{j}, \qquad \text{(by (\ref{hausd-abinEl1}))}
\end{align*}
which is a contradiction. 

Next, observe that $\{\psi^{-1}(E_{l})\}_{l \in \N}$ is necessarily a cover of $D_{j}$ so
\begin{align*}
\mathcal{H}_{X}^{m}(D_{j}) &\le \sum_{l=1}^{\infty} \omega_{m} \left( \frac{\diam_{X}(\psi^{-1}(E_{l}))}{2}\right)^{m} \\
	&= \sum_{l=1}^{\infty} \omega_{m} \left( \frac{\diam_{Y}(E_{l})}{2}\right)^{m}. 
		\qquad \text{(by (\ref{hausd-diamequal}))}\\
\end{align*}
Taking the infimum over all covers of $C_{j}$ with diameters less than $\delta$ gives
\[
\mathcal{H}_{X}^{m}(D_{j}) \le \mathcal{H}_{Y,\delta}^{m}(C_{j})
\]
then taking the limit as $\delta \to 0$ shows
\[
\mathcal{H}_{X}^{m}(D_{j}) \le \mathcal{H}_{Y}^{m}(C_{j})
\]
which proves the claim (\ref{hausd-DCineq}). 

To finish, we take the limit in (\ref{hausd-DCineq}) as $j \to \infty$ and use (\ref{hausd-limit1}) and (\ref{hausd-limit2}) to complete the proof. 
\end{proof}

%%%%%%%%%%%%%%%%%%%%%%%%%%%%%%%
%Sewing to Pulled Strings
\section{Sewn Manifolds converging to Pulled Strings}

In this section we consider a sequences of sewn manifolds being sewn increasingly tightly and prove they converge in the Gromov-Hausdorff and Intrinsic Flat sense to metric spaces with pulled strings.   

To be more precise, we consider the following sequences
of {\em increasingly tightly sewn} manifolds:

\begin{defn}\label{def-seq}
Given a single Riemannian manifold, $M^3$, with a curve, 
$A_0=C([0,1])\subset M$, with a tubular neighborhood
$A=T_a(A_0)$ which is Riemannian isometric to a tubular neighborhood of a compact set $V \subset \mathbb{S}^3_K$, in a 
standard sphere of constant sectional curvature $K$,
satisfying the hypothesis of Proposition~\ref{sewn-curve}.
We can construct its sequence of increasingly tightly sewn
manifolds, $N_j^3$, by applying Proposition~\ref{sewn-curve} taking 
$\epsilon=\epsilon_j \to 0$, $n=n_j \to \infty$, and $\delta=\delta_j\to 0$ to create each sewn manifold, $N^3=N_j^3$ and the edited regions $A_{\delta}'=A_{\delta_{j}}'$ which we simply denote by $A_{j}'$.   
This is depicted in Figure~\ref{fig-sewn-seq}.
Since these sequences $N_j^3$ are created using Proposition~\ref{sewn-curve}, they have positive scalar curvature whenever $M^3$ has positive scalar curvature,
and $\partial N_j^3=\partial M^3$ whenever $M^3$ has a
nonempty boundary.
\end{defn}

\begin{figure}[htbp]
\begin{center}
\includegraphics[width=1.5in]{sew1.png}
\includegraphics[width=1.5in]{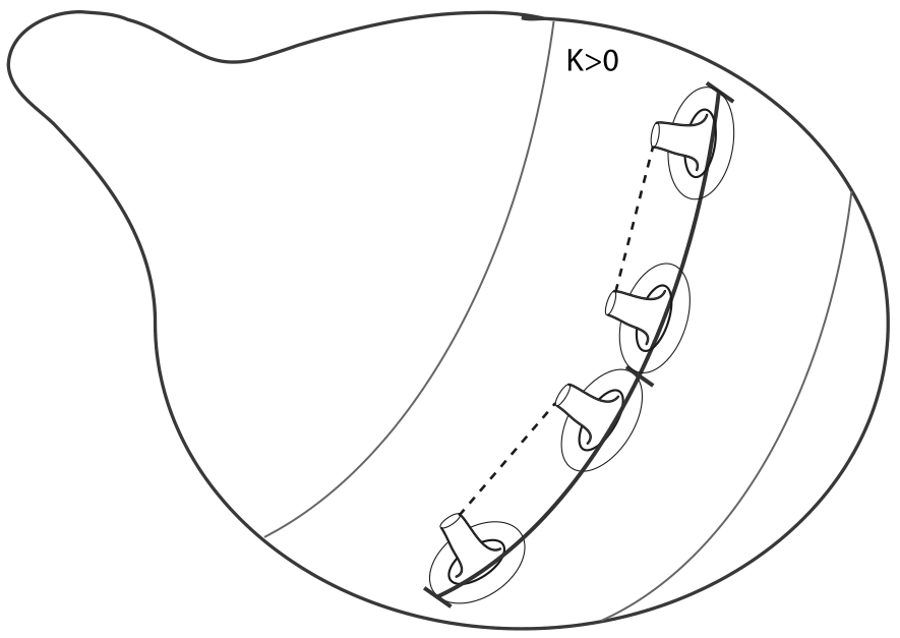}
\includegraphics[width=1.6in]{sew3.png}
\caption{A sequence of increasingly tightly sewn manifolds.}
\label{fig-sewn-seq}
\end{center}
\end{figure}

In this section we prove Lemma~\ref{L:almostiso}, Lemma~\ref{L:mGHconv} and Lemma~\ref{L:flatconv}, which immediately imply the following theorem:

\begin{thm}\label{thm-seq-sewn}
The sequence $N_j^3$ as in Definition~\ref{def-seq} converges in the
Gromov-Hausdorff sense 
\be
N_j^{3} \GHto N_\infty,
\ee
the metric measure sense
\be
N_j^{3} \mGHto N_\infty,
\ee
and the intrinsic flat sense
\be
N_j^{3} \Fto N_\infty,
\ee
where $N_\infty$ is the metric space created by pulling the
string, $A_0=C([0,1])\subset M$, to a point as in Proposition~\ref{pulled-string}.  
\end{thm}

In fact our lemmas concern more general sequences of manifolds which
are constructed from a given manifold $M$ and
{\em scrunch} a given compact set $K\subset M$ down to a point as follows:

\begin{defn}\label{def-seq-down}
Given a single Riemannian manifold, $M^3$, with a compact set, 
$A_0\subset M$.  A sequence of manifolds,
\be
N_j^3= (M^3 \setminus A_{\delta_j})\disjointunion A'_{\delta_j} 
\ee
is said to scrunch $A_0$ down to a point
if $A_{\delta}=T_{\delta}(A_0)$ and
$A'_\delta$ satisfies:
\be\label{sewn-curve-tubular'}
(1-\epsilon)\vol(A_{\delta}) \le \vol(A_{\delta}')\le \vol(A_{\delta})(1+\epsilon) 
\ee
and
\be\label{sewn-curve-vol'}
(1-\epsilon) \vol(M^3)\le \vol(N^3) \le \vol(M^3) (1+\epsilon)
\ee
and
\be \label{sewn-curve-diam'}
\diam(A_{\delta}')\le H
\ee
where $\epsilon=\epsilon_j \to 0$ and where $H=H_j \to 0$ and $2\delta_j<H_j$.
\end{defn}

Note that by Proposition~\ref{sewn-curve},
a sequence of increasingly tightly sewn manifolds sewn
along a curve $C([0,1])$ as in Definition~\ref{def-seq} is a sequence
of manifolds which scrunches $A_0=C([0,1])$ down to a point as 
in Definition~\ref{def-seq-down}.   So we will prove lemmas about
sequences of manifolds which scrunch a compact set and then apply them
to prove Theorem~\ref{thm-seq-sewn} in the final subsection of this section.

\subsection{Constructing Surjective maps to the limit spaces}

Before we prove convergence of the scrunched sequence of manifolds to the pulled thread space, we construct surjective maps from the sequence to the proposed limit space.

% lemma surj maps
\begin{lem}\label{L:surjectivemaps}
Given $M^3$ a compact Riemannian manifold (possibly with boundary) and a smooth embedded compact zero to three 
dimensional submanifold $A_0\subset M^3$ (possibly with boundary),
and $N_j$ as in Definition~\ref{def-seq-down}.
Then for $j$ sufficiently large
there exist surjective Lipschitz maps
\be\label{surjlipmaps4}
F_j: N_j^3 \to N_\infty \textrm{ with } \Lip(F_j) \le 4
\ee
where $N_\infty$ is the metric space created by taking $M^3$ and 
pulling $A_0$ to a point $p_0$
as in Lemmas~\ref{pulled-subset-1}-~\ref{pulled-subset-2}. 
\end{lem}

Note that when $A_0$ is the image of a curve, $N_\infty$, is a pulled thread space as in Proposition~\ref{pulled-string}.

\begin{proof}
First observe that by the construction in Definition~\ref{def-seq-down}
there are maps
\be
P_j: M^3 \to N_\infty
\ee
which are Riemannian isometries on regions which avoid $A_0$
and map $A_0$ to $p_0$.  These define Riemannian isometries 
\be
P_j: N_j^3 \setminus A_j' \tilde{=}M^3 \setminus T_{\delta_j}(A_{0})
\to N_\infty^3 \setminus T_{\delta_j}(p_0).
\ee
In addition sufficiently small
balls lying in these regions are isometric to convex balls in $M^3$.

Observe also that for $\delta>0$ sufficiently small, the exponential
map:
\be
exp: \{(p,v):\,\,p\in A_0,\,\, v\in V_p \,\,|v|<2\delta\} \to T_{2\delta}(A_0)
\ee
is invertible where
\be
V_p=\{v\in T_pM:\,\, d_M(exp_p(tv),p)=d_M(exp_p(tv),A_0)\}.
\ee
Taking $\delta=\delta_{A_0}>0$ even smaller (depending on the submanifold $A_0$), 
we can guarantee that $\forall v_i\in V_p, |v_i|<2\delta_{A_0}, t_i\in(0,1)$ we have
\be \label{use-smooth}   
d_M(exp_{p_1}(t_1v_1), exp_{p_2}(t_2v_2))
\le  2 d_M(exp_{p_1}(v_1), exp_{p_2}(v_2)) + 2|t_1-t_2|.
\ee

This is not true unless $A_0$ is a smooth embedded compact submanifold with either no boundary or a smooth boundary.   
 
Define $F_j: N_j^3 \to N_\infty$ as follows:
\be\label{Fj1}
F_j(x)=P_j(x) \qquad \forall x \in N_j^3\setminus T_{\delta_j}(A_j')
\ee
and
\be\label{Fj2}
F_j(x)=p_0 \qquad \forall x\in A_j'.
\ee
Between these two regions we take
\be\label{Fj3}
F_j(x)=f_j(P_j(x)) \qquad \forall x \in T_{\delta_j}(A_j')\setminus A_j'
\ee
where $f_j: N_\infty \to N_\infty$ is a surjective map:
\be
f_j: Ann_{p_0}(\delta_j, 2\delta_j) \to B_{2\delta_j}(p_{0})\setminus \{p_0\}
\ee
which takes a point $q$ to 
\be
f_j(q)=\gamma_q\left( (d_{N_\infty}(p_0, q)- \delta_j) /\delta_j \right)
\ee
where $\gamma_q$ is the unique minimal geodesic from $\gamma_q(0)=p_0$
to $\gamma_q(1)=q$.   
Here we are assuming $\delta_j<\delta_{A_0}$.   So 
\be
d_{N_\infty}(p_0, P_j(x))=d_{M^3}(A_0, x)
\ee
and
\be
\gamma_q(t)=P_j(exp_{q'}(tv')) \textrm{ where } P_j(exp_{q'}(v'))=q.
\ee
In particular for $x\in \partial T_{\delta_j}(A_j')$,
\be
f_j(P_j(x))=\gamma_{P_j(x)}((2\delta_j-\delta_j)/\delta_j)=\gamma_{P_j(x)}(1)=P_j(x)
\ee
and for $x\in \partial A_j'$,
\be
f_j(P_j(x))=\gamma_{P_j(x)}((\delta_j-\delta_j)/\delta_j)=\gamma_{P_j(x)}(0)=p_0
\ee
so that $F_j$ is continuous.

We claim
\begin{eqnarray}
\Lip(F_j)&=&0 \textrm{ on } A_j' \\
\Lip(F_j) &\le& 4 \textrm{ on } T_{\delta_j}(A_j')\setminus A_j' \label{caseLip4}\\
\Lip(F_j)&=&1\textrm{ on } N_j \setminus T_{\delta_j}(A_j').
\end{eqnarray}
Only the middle part is difficult.   By
the definition of $d_{N_\infty}$ we have the following two
possibilities
\begin{eqnarray}
\textrm{Case I: \,} &&d_{N_\infty}(q_1,q_2)=
d_{M}(P_j^{-1}(q_1), P_j^{-1}(q_2)) \\
\textrm{Case II:} && d_{N_\infty}(q_1, q_2)=
d_M(P_j^{-1}(q_1), A_0)+
d_M(P_j^{-1}(q_2), A_0).
\end{eqnarray}
In Case II we see that the minimal geodesic
from $q_1$ to $q_2$ passes through $p_0$.  Since
$f_j(q_1)$ and $f_j(q_2)$ lie on this geodesic, we have
\be
d_{N_\infty}(f_j(q_1), f_j(q_2))\le d_{N_\infty}(q_1,q_2).
\ee
In Case I we apply (\ref{use-smooth}) with
\be
t_i=(d_{M}(P_j^{-1}(q_i),A_0)- \delta_j) /\delta_j
\ee
because $t_{i} \in (0,1)$ due to (\ref{caseLip4}) so that by the reverse triangle inequality
\begin{eqnarray}
|t_1-t_2|&=&|d_{M}(P_j^{-1}(q_1),A_0)-d_{M}(P_j^{-1}(q_2),A_0)|/\delta_j \\
&\le & d_{M}(P_j^{-1}(q_1),q_2)/\delta_j \\ 
&\le & d_{N_\infty}(q_1,q_2)
\end{eqnarray}
to see that
\begin{eqnarray}
d_{N_\infty}(f_j(q_1), f_j(q_2))&\le&
d_{M}(P_j^{-1}(f_j(q_1)), P_j^{-1}(f_j(q_2))) \\
&\le& 2 d_M(P_j^{-1}(q_1),P_j^{-1}(q_2))+ 2|t_1-t_2| \textrm{ by (\ref{use-smooth})},\\
&\le& 2 d_{N_\infty}(q_1,q_2) + 2|t_1-t_2| \textrm{ by Case I hypothesis,}\\
&\le & 4d_{N_\infty}(q_1,q_2).
\end{eqnarray}
This gives our claim.

We claim $\Lip(F_j)\le 4$ everywhere.   Given $x_1,x_2\in N_j^3$,
we have  a minimizing geodesic $\eta:[0,1]\to N_j$
such that $\eta(0)=x_1$ and $\eta(1)=x_2$.   Then
\be
d_{N_\infty}(F_j(x_1), F_j(x_2)) \le L(F_j \circ \eta).
\ee
Since $|(F_j\circ \eta)'(t)| \le 2|\eta'(t)|$ by our localized
Lipschitz estimates and because the function $F_j$ is
continuous, we are done.
\end{proof}

\subsection{Constructing Almost Isometries}

See Section~\ref{sect-GH-back} for a review of the Gromov-Hausdorff
distance.

%IMPROVE THIS TO APPLY IN PAPER 2 TO A VARIABLE SEQUENCE?
%BY MAKING IT AN ESTIMATE FOR A SINGLE j
\begin{lem}\label{L:almostiso}
Given $N_j^{3}$ as in Definition~\ref{def-seq-down},
the maps $F_j: N_j^3 \to N_\infty$ defined in (\ref{Fj1})-(\ref{Fj3}) in the proof of Lemma~\ref{L:surjectivemaps} are $H_j$-almost isometries with $\lim_{j\to \infty}H_j=0$.
Thus
\be\label{NjGHtoNinfty}
N_j \GHto N_\infty.
\ee
\end{lem}

\begin{proof}
Before we begin the proof recall that
\be
\diam(A_j')\le H_j \to 0
\ee
in (\ref{sewn-curve-diam'}) of Definition~\ref{def-seq-down}.

By Theorem~\ref{almost-isom} of Gromov, to prove (\ref{NjGHtoNinfty}) it suffices to show that $F_j$ are $H_j$-almost isometries.
To see this, examine $x,y\in N_j$ and join them by a minimizing curve $\sigma:[0,1]\to N_j$. 

If $\sigma[0,1]\subset N_j \setminus A'_j$, then 
by (\ref{Fj1}) we have
\be
L(\sigma)=L(F_j\circ \sigma) 
\ee
and so
\be
d_{N_j}(x,y)\ge d_{N_\infty}(F_j(x),F_j(y)).
\ee
Otherwise we have
\begin{eqnarray}
d_{N_j}(x,y) &\ge& d_{N_j}(x, A'_j) + d_{N_j}(y, A'_j) 
	\qquad T_{\delta_j}(A'_j) \text{ to } A'_j \\
&=& d_{N_\infty}(F_j(x), B_{\delta_j}(p_0)) + d_{N_\infty}(F_j(y), B_{\delta_j}(p_0))\\
&= &d_{N_\infty}(F_j(x), p_0)-\delta_j + d_{N_\infty}(F_j(y), p_0)-\delta_j \\
&\ge & d_{N_\infty}(F_j(x), F_j(y)) -2\delta_j. 
\end{eqnarray}

Next we join $F_j(x)$ to $F_j(y)$ by a minimizing curve $\gamma$.
If $\gamma[0,1]\subset N_\infty \setminus B_{\delta_j}(p_0)$ then there
is a curve $\eta$ such that $\gamma=F_j\circ\eta$ with
$\eta[0,1]\subset N_j \setminus A'_j$ 
and so by (\ref{Fj1})
\be
d_{N_j}(x,y)\le L(\eta)=L(\gamma) =d_{N_\infty}(F_j(x), F_j(y)).
\ee
Otherwise we have
\begin{eqnarray}
d_{N_j}(x,y) 
&\le & d_{N_j}(x, A_j')+\diam(A_j') + d_{N_j}(y, A_j') \\ 
%&\le & d_{N_j}(x, T_{\delta_j}(A_j'))+2 \delta_j+ \diam(A_j') + d_{N_j}(y, T_{\delta_j}(A_j')) \\ 
&\le & d_{N_j}(x, A_j')+H_{j}+ d_{N_j}(y, A_j') \\ 
&=& d_{N_\infty}(F_j(x), B_{\delta_j}(p_0)) + d_{N_\infty}(F_j(y), B_{\delta_j}(p_0)) + H_j\\ 
&\le & L(\gamma) +H_j =d_{N_\infty}(F_j(x),F_j(y)) + H_j. 
\end{eqnarray}
Hence, $F_j$ is an $H_j$ isometry since $2\delta_{j}<H_{j}$.
\end{proof}

\subsection{Metric Measure Convergence}

Recall metric measure convergence as reviewed in Section~\ref{sect-mm-back}.

%IMPROVE THIS TO APPLY IN PAPER 2 TO A VARIABLE SEQUENCE?
%BY MAKING IT AN ESTIMATE FOR A SINGLE j
\begin{lem}\label{L:mGHconv}
Given $N_j^3 \to N_\infty$ as in Lemma~\ref{L:surjectivemaps}
endowed with the Hausdorff measures, then we have metric measure convergence if $A_0$ has $\mathcal{H}^{3}$-measure $0$.
\end{lem}

\begin{proof}
Recall the maps $F_j: N_j^3 \to N_\infty$ defined in (\ref{Fj1})-(\ref{Fj3}) in the proof of Lemma~\ref{L:surjectivemaps}.   We
need only show that for almost every $p\in N_\infty$ and
for almost every $r<r_p$ sufficiently small we have
\be
\mathcal{H}^3(B(p,r))=\lim_{j\to \infty}\mathcal{H}^3(B(p_j,r))
\ee
where $F_j(p_j)=p$ and that for any sequence $p_{0j}\to p_0$
we have $r_0$ sufficiently small that for all $r<r_0$
\be
\mathcal{H}^3(B(p_0,r))=\lim_{j\to \infty}\mathcal{H}^3(B(p_{0j},r)).
\ee

In fact take any $p\neq p_0$ in $N_{\infty}$ and
choose 
\be
r<r_p<d_{N_\infty^3}(p,p_0)/2.
\ee
Then for $j$ large enough that $\delta_j< r_p$ we have
\be
B(p,r)\cap B(p_0,\delta_j)=\emptyset.
\ee
Thus
\be
B(p_j,r)\cap A_j'=\emptyset.
\ee
Thus by (\ref{Fj1}), $F_j$ is an isometry from $B(p_j,r) \subset N_j^3$
onto $B(p,r)\subset N_\infty$ 
and so we have
\be
\mathcal{H}^3(B(p,r))=\mathcal{H}^3(B(p_j,r))\qquad \forall r<r_p.
\ee

Next we examine $p_0$.  Observe that by (\ref{pulledsetHm})  
\be
\mathcal{H}_{N_\infty}^3(B(p_0,r)) = 
\mathcal{H}_{M}^3(T_r(A_0))-\mathcal{H}_{M}^3(A_0) =
\vol_M(T_r(A_0)\setminus A_0).
\ee
For any $p_{0,j}\to p_0$, we have by (\ref{surjlipmaps4}) 
\be
r_j=d_{N_j}(p_{0,j}, A_j')\le 4d_{N_\infty}(F_j(p_{0,j}), p_0) \to 0
\ee 
Thus
\be
B(p_{0,j},r) \subset T_{r+r_j}(A_j').
\ee
So
\begin{eqnarray}
\vol_{N_j}(B(p_{0,j},r))&\le& \vol_{N_j}(T_{r+r_j}(A_j'))\\
&\le & \vol_{N_j}(T_{r+r_j}(A_j')\setminus A_j') +\vol_{N_j}(A_j')\\
&=& \vol_M\left(T_{r+r_j+\delta_j}(A_0) 
                           \setminus T_{\delta_j}(A_0)\right)
          +\vol_{N_j}(A_j').
\end{eqnarray}        
  
Thus
\begin{eqnarray}
\limsup_{j\to\infty}\vol_{N_j}(B(p_{0,j},r))
&\le& \vol_M\left(T_{r}(A_0) \setminus A_0\right)
          +\limsup_{j\to\infty} \vol_{N_j}(A_j') \\
&=& \mathcal{H}^3(B(p_0,r))
\end{eqnarray}
since we claim that 
\be\label{volAj'to0}
\lim_{j\to\infty}\vol_{N_j}(A_j') = 0. 
\ee
This follows because $\epsilon_{j} \to 0$ and
 (\ref{sewn-curve-tubular'}) implies
\be
(1-\epsilon_j) \vol_M(A_{\delta_{j}}) \le 
\vol_{N_j}(A_j') \le (1+\epsilon_j)\vol_{M}(A_{\delta_{j}}).
\ee
The assumption that $\mathcal{H}^{3}(A_{0})=0$ then implies (\ref{volAj'to0}) after taking the limit. 

Similarly, we have for $j$ sufficiently large
\be
T_{r-H_j-r_j}(A_j') \subset B(p_{0,j},r).
\ee
So
\begin{eqnarray}
\vol_{N_j}(B(p_{0,j},r))&\ge& \vol_{N_j}(T_{r-H_j-r_j}(A_j'))\\
&= & \vol_{N_j}(T_{r-H_j-r_j}(A_j')\setminus A_j') +\vol_{N_j}(A_j')\\
&=& \vol_M\left(T_{r-H_j-r_j+\delta_j}(A_0) 
                           \setminus T_{\delta_j}(A_0)\right)
          +\vol_{N_j}(A_j').
\end{eqnarray}          
Thus
\begin{eqnarray}
\liminf_{j\to\infty}\vol_{N_j}(B(p_{0,j},r))
&\ge& \vol_M\left(T_{r}(A_0) \setminus A_0\right)
          +\liminf_{j\to\infty}\vol_{N_j}(A_j')\\
&=& \mathcal{H}^3(B(p_0,r)), \,\,\, \textrm{ by (\ref{volAj'to0})}
\end{eqnarray}
which completes the proof.
\end{proof}

\subsection{Intrinsic Flat Convergence}

For a review of intrinsic flat convergence 
see Section~\ref{sect-SWIF-back}.

%IMPROVE THIS TO APPLY IN PAPER 2 TO A VARIABLE SEQUENCE?
%BY MAKING IT AN ESTIMATE FOR A SINGLE j
\begin{lem}\label{L:flatconv}
Let $N_j^3 \GHto N_\infty$ be exactly as in Lemma~\ref{L:surjectivemaps} and Lemma~\ref{L:almostiso} 
where we assume $M$ is compact and we have a compact set,
$A_0\subset M\setminus \partial M$.
Then there exists an integral current space $N$ such 
that $\bar{N}$ is isometric to $N_\infty$ and 
\be
N_j \Fto N.
\ee
and when $A_0$ has Hausdorff measure $0$
\be \label{L:massconv}
\mass(N_j) \to \mass(N)=\mathcal{H}^3(N).   
\ee
When $A_0=C([0,1])$ then $N=N_\infty$.
\end{lem}

\begin{proof}
By (\ref{sewn-curve-vol'}), we have uniformly bounded volume
\be
\vol(N_j^3) \le 2 \vol(M^3). 
\ee
Since $\partial N_j^3=\partial M^3$, we have
uniformly bounded boundary volume
\be
\vol(\partial N_j^3) =\vol(\partial M^3).
\ee
Combining this with Lemma~\ref{L:almostiso} and Theorem~\ref{GH-to-flat},
there exists an integral current space $N$ possibly $N={\bf{0}}$ such that a subsequence
\be\label{subseq1}
N_j \Fto N. 
\ee

We claim that $N\neq{\bf{0}}$.    If not, then by the final line in
Lemma~\ref{balls-converge}, for any sequence $p_j\in N_j$
and almost every $r$, $S(p_j,r) \Fto {\bf{0}}$.  However, taking $p_j$ and
$r$ such that 
\be
B(p_j,r)\subset N_j^3 \setminus A_j' 
\ee
we know there is some $p \in M^3$ with $B(p,r) \subset N_\infty\setminus\{p_0\}$ that $d_{\mathcal{F}}(S(p_j,r), S(p,r))=0$ for $p\in M^3$,
so $S(p_j,r) \Fto S(p,r) \neq {\bf{0}}$ which is a contradiction.

By Theorem~\ref{Flat-Arz-Asc}, we know that after possibly taking a subsequence 
we obtain a limit map 
\be
F_\infty: N \to N_\infty. 
\ee

We claim that $F_\infty$ is distance preserving.   Let $p,q\in N$.   By
Theorem~\ref{to-a-limit}, we have $p_j,q_j \in N_j$ converging to $p,q$
in the sense of Definition~\ref{point-conv}, i.e.
\be
d_{N_j}(p_j, q_j ) \to d_{N}(p,q).
\ee
Since the $F_j$ are $\epsilon_j$-almost isometries and $\epsilon_j \to 0$, we have
\be
d_{N_\infty}(F_j(p_j), F_j(q_j)) \to d_N(p,q).
\ee
By the definition of $F_\infty$ we have $F_j(p_j)\to F_\infty(p)$
and $F_j(q_j)\to F_\infty(q)$.  
Thus 
\be
d_{N_\infty}(F_\infty(p), F_\infty(q)) = d_N(p,q).
\ee

We claim that $F_\infty$ maps onto at least $N_\infty \setminus\{p_0\}$.   
Let $x \in N_\infty\setminus\{p_0\}$.   Since $F_j$ are
surjective, there 
exists $x_j\in N_j$ such that $F_j(x_j)=x$.   Since $x\neq p_0$,
we may define 
\be
r=\min\{d_{N_\infty}(x,p_0)/3, \textrm{ConvexRad}_M(x)\}  
\ee
where $\textrm{ConvexRad}_M(x)$ is the convexity radius about $x$ viewed
as a point in $M$.  
Then there exists $j$ sufficiently large such that $\delta_j<r$ so that
\be
B(x_j, r) \subset N_j \setminus  T_{\delta_j}(A_j').
\ee
Furthermore, these balls are isometric to the convex ball
$B(x,r)\subset M^3$.  

So
\be
d_{\mathcal{F}}(S(x_j,r), {\bf{0}}) =d_{\mathcal{F}}(S(x,r), {\bf{0}}) >0. 
\ee
Thus by Theorem~\ref{B-W-BASIC}
with $h_0=d_{\mathcal{F}}(S(x,r), {\bf{0}})$, and $N_j \Fto N$, 
a subsequence of the $x_j$ converges to $x_\infty \in N$.   
By the definition of $F_\infty$, we have $F_j(x_j) \to F_\infty(x_\infty) \in N_\infty$. 
But since $F_j(x_j)=x$ it follows that $F_\infty(x_\infty)=x$, hence $F_\infty$ maps onto $N_\infty\setminus p_0$.  

Taking the metric completions of $N$ and $N_\infty \setminus \{p_0\}$, 
we have an isometry
\be
F_\infty: \bar{N} \to N_\infty.
\ee

Since $N_j$ are Riemannian manifolds, 
\be
\mass(\Lbrack N_j \Rbrack)=\vol(N_j)=\mathcal{H}^3(N_j).
\ee   
By the lower semicontinuity of mass
and the metric measure convergence of $N_j$ to $N$ we know that
\be
\mass(\Lbrack N_{\infty} \Rbrack)\le \liminf_{j\to\infty}\mass(\Lbrack N_j \Rbrack)=\mathcal{H}^3(N).
\ee
On the other hand by (\ref{mass-current-space})
\be
\mass(\Lbrack N_{\infty} \Rbrack)\ge \mathcal{H}^3(N)
\ee
because almost every tangent cone is Euclidean and it has
integer weight everywhere.   Thus we have (\ref{L:massconv}).
In fact equality in these inequalities implies that $N$ has weight one everywhere.
  
Recall that the set of an integral current space only includes points of
positive density.  Since 
\be
\liminf_{r\to 0} \frac{\vol_{N_\infty}(B(p_0,r))}{r^3}
=\liminf_{r\to 0} \frac{\vol_{M}(T_r(A_0)\setminus A_0)}{r^3}
\ee
Thus $N$ is isometric to $N_\infty$ when this
liminf is positive and $N$ is isometric to $N_\infty\setminus \{p_0\}$
when this liminf is $0$.  When $A_0=C([0,1])$ is a curve
in a 3 dimensional Riemannian manifold we have
\be
 \liminf_{r\to 0} \frac{\vol_{M}(T_r(A_0)\setminus A_0)}{r^3}
 = \liminf_{r\to 0} \frac{\pi r^2 L(C)}{r^3} = + \infty >0.
\ee
Thus $N$ is isometric to $N_\infty$.

Thus $N$ does not depend on the subsequence in (\ref{subseq1}) and in fact the original sequence (given a consistent orientation)
converges in the intrinsic flat sense to $N$.   
\end{proof}

\subsection{The proof of Theorem~\ref{thm-seq-sewn}.}

\begin{proof}
In Proposition~\ref{sewn-curve} we show that given any
$\epsilon_j \to 0$ we can find $n_j \to \infty$ and $\delta_j \to 0$
so fast that $\delta_jn_j \to 0$ and we have $h(\delta_j)n_j \to 0$
as well such that the sewn manifolds:
\be
N_j^3 = (M^3 \setminus A_{\delta_j} )\disjointunion A_{\delta_j}',
\ee
satisfy:
\be
(1-\epsilon)\vol(A_{\delta}) \le \vol(A_{\delta}')\le \vol(A_{\delta})(1+\epsilon) 
\ee
and
\be
(1-\epsilon) \vol(M^3)\le \vol(N^3) \le \vol(M^3) (1+\epsilon)
\ee
and
\be 
\diam(A_{\delta}')\le H(\delta)= L(C)/n + (n+1)\, h(\delta)+(5n+2)\, \delta.
\ee
where
\be%\label{sewn-curve-Hto0}
\lim_{\delta\to 0} H(\delta)=0 \textrm{ uniformly for } K\in (0,1].
\ee
Thus we have a sequence $N_j$ which is scrunching a set
$A_0=C([0,1])$ to a point as in Definition~\ref{def-seq-down}

Lemma~\ref{L:almostiso} implies that
\be
N_j \GHto N_\infty
\ee
where $N_\infty$ is the pulled string space.
Lemma~\ref{L:mGHconv} implies we have metric measure to $N_\infty$
convergence because $A_0=C([0,1])$ has $\mathcal{H}^{3}$-measure $0$.

Lemma~\ref{L:flatconv} implies that
\be
N_j \Fto N_\infty
\ee
and 
\be %\label{L:massconv}
\mass(N_j) \to \mass(N_\infty)=\mathcal{H}^3(N),
\ee
completing the proof of Theorem~\ref{thm-seq-sewn}. 
\end{proof}

%%%%%%%%%%%%%%%%%%%%%%%%%%%%%%%%%%
% Examples Constructed from Spheres
\section{Sewing a Sphere to Obtain our Limit Space}

Here we construct the specific example of a sequence of manifolds with
positive scalar curvature that converges to a limit space which fails
to have generalized nonnegative scalar curvature as discussed in the 
introduction.  More specifically:

% sphere geodesic example
\begin{example}\label{sphere-geod}
We define a sequence $N_j^3$ of manifolds with positive scalar curvature
constructed from the standard $\mathbb{S}^3$ sewn along a closed geodesic $C:[0,1]\to \mathbb{S}^3$
with $\delta=\delta_j \to 0$ as in Proposition~\ref{sewn-curve}.   Then
by Theorem~\ref{thm-seq-sewn} we have
\be\label{sphere-geod-conv}
N_j^3 \mGHto N_\infty \textrm{ and } N_j^3 \Fto N_\infty
\ee
where $N_\infty$ is the metric  space created by taking the
standard sphere and pulling the geodesic to a point as in Proposition~\ref{pulled-string}.  By Lemma~\ref{bad-point} below we see that
at the pulled point $p_0\in N_\infty$ we have (\ref{sphere-vol-bad}).  Thus we
have produces a sequence of three dimensional manifolds with positive
scalar curvature converging to a limit space which fails to satisfy generalized scalar curvature defined using limits of volumes of balls as in (\ref{sphere-vol}).
\end{example}

\begin{rmrk}
Note that with $\delta_j \to 0$, the neck in the center of the tunnels
has a rotationally symmetric minimal surface whose area is $\le 4\pi\delta_j^2$
which converges to $0$.  So this sequence, and in fact any sewn sequence 
created as in Definition~\ref{def-seq}, has $\textrm{MinA}(N_j) \to 0$.
\end{rmrk}

\begin{lem}\label{bad-point}
At the pulled point $p_0\in N_\infty$ of Example~\ref{sphere-geod}
we have 
\be\label{sphere-vol-bad'}
\lim_{r\to0} \left( \frac{\vol_{\E^3}(B(0,r)) - \vol_{N_\infty}(B(p_0,r))}{r^2 \vol_{\E^3}(B(0,r))}\right)=-\infty.
\ee
\end{lem}

\begin{proof}
First, observe that 
\begin{eqnarray}
\vol_{N_\infty}(B(p_0,r)) &=& \mathcal{H}_{N_\infty}^3\left(B(p_0,r)\right) \\
	&=& \mathcal{H}_{N_\infty}^3\left(B(p_0,r) \setminus \{p_0\}\right) \\
	&=& \mathcal{H}_{\mathbb{S}^3}^3\left(\,T_r(C([0,1]))\,\right).
\end{eqnarray}
Since $C([0,1])$ is a closed geodesic of length $2\pi$
in a three dimensional sphere, we have
\be
\lim_{r\to 0}\,\, \frac{\,\mathcal{H}_{\mathbb{S}^3}^3\left(\,T_r(C([0,1]))\,\right)\,}{  2\pi (\pi r^2)}\,\,\,=\,\,\,1.
\ee
Thus
\be
\lim_{r\to 0} \frac{\vol_{\E^3}(B(0,r)) - \vol_{N_\infty}(B(p_0,r))}{r^2 \vol_{\E^3}(B(0,r))}
=\lim_{r\to 0} \frac{(4/3)\pi r^3 -  2\pi (\pi r^2)}{(4/3)\pi r^5 }
=-\infty
\ee
as claimed.
\end{proof}

%%%%%%%%%%%%%%%%%%%%%%%%%%%%%%%%%%%%%%%%%%%
% APPENDIX 

\section{Appendix: Short tunnels with Positive Scalar Curvature\\
by Jorge Basilio and J\'ozef Dodziuk}

There is a deep connection between the geometry of Riemannian manifolds $M^n$ with positive scalar curvature and surgery theory. The subject began with the surprising discovery by Gromov and Lawson \cite{Gromov-Lawson-tunnels} (for $n \ge 3$) and Schoen and Yau \cite{Schoen-Yau-tunnels} that a manifold obtained via a surgery of codimension 3 from a manifold $M^n$ with a metric of positive scalar curvature may also be given a metric with positive scalar curvature. 
%The methods of Schoen and Yau \cite{Schoen-Yau-tunnels} rely on the regularity theory of minimal surfaces and therefore require the dimension of their manifolds less than or equal to seven. On the other hand, the techniques used by Gromov and Lawson \cite{Gromov-Lawson-tunnels} are geometric and thus work for any dimension greater than or equal to three (and surgeries of codimension at least 3). 
The key to the tunnel construction of \cite{Gromov-Lawson-tunnels} is defining a curve $\gamma$ which begins along the vertical axis then bends upwards as it moves to the right and ends with a horizontal line segment, cf.\ Figure \ref{F:gammak} below. The tunnel then is  the surface of revolution determined by $\gamma$. %See Figure~\ref{F:tunnels}. 
We note that the ``bending argument'' has attracted some attention (See \cite{Rosenberg-Stolz-PSC2001}).

As the goals of the surgery theory were topological in nature Gromov and Lawson did not estimate with diameters or volumes of these tunnels. Indeed, the tunnels they constructed may be thin but long (See \cite{Gromov-Lawson-tunnels2}). 
To build sewn manifolds we need tunnels with diameters shrinking to zero as the size of the original balls decreases to zero (see (\ref{TL-diameterU}), (\ref{TL-diameterUorderdelta}) (\ref{TL-tunnellengthtozero})). Therefore, we prove Lemma~\ref{tunnellemma} to obtain a refinement of the Gromov and Lawson construction  showing the existence of tiny (in sense of (\ref{TL-volumeestU})) and arbitrarily short tunnels with a metric of positive scalar curvature.

% BEGIN PROOF
%
\begin{proof}[Proof of Lemma~\ref{tunnellemma}] % ADDED! Doesn't seem to show up in document ...

To aid the reader, we provide a summary of our proof and introduce additional notation.

% OUTLINE OF PROOF
\subsection{Outline of Proof of Lemma~\ref{tunnellemma}.}\label{subsection:outline} 
To aid the reader, we provide a summary of our proof and introduce additional notation.

% STEP 1 outline		
\noindent 
\textbf{Step 1: Setup and notation.} 
Let $\epsilon>0$ be given. We shall specify $0<\delta_0<\delta/2$ below.

Given that $B_1=B(p_1,\delta/2) \subset M^3$ has constant sectional curvature $K>0$, we may choose coordinates so that it is realized as a hypersurface of revolution. This is also true for $B(p_1,\delta_0) \subset B_1$ for $0<\delta_0<\delta/2$ centered at the same $p_1$. Thus, $B(p_1,\delta_0)$ is a hypersurface of revolution $U_{\gamma_0}'$ with the induced metric in $\R^4$ determined by revolving a segment of the circle $\gamma_0$ in the $(x_0,x_1)$-plane about the $x_0$-axis. 
We set things up so that the vertical $x_1$-axis corresponds to boundary points of $B(p_1,\delta_0)$. 
We then proceed as Gromov and Lawson to deform $\gamma_0$ away from vertical axis bending it upwards as we move to the right and ending with an arbitrarily short horizontal line segment. We call this curve $\gamma$, cf.\ Figure~\ref{F:gammak}. 
%\begin{figure}[h]
%	\begin{center}
%		\includegraphics[scale=.2]{gamma.png} % FIGURE
%	\end{center}
%	\caption{The curve $\gamma$.}\label{F:finalcurvegamma}
%\end{figure}
%
The curve $\gamma$ begins exactly as $\gamma_0$ so that we may attach the corresponding hypersurface onto the larger $B(p_1,\delta/2)$ in a natural way. 
We  do exactly the same for $B_2 \subset M^3$ and identify the two hypersurfaces along their common boundary, i.e the ``tiny neck,'' forming $2U_{\gamma}'=U_{\gamma}' \disjointunion U_{\gamma}'$. 
We then define the tunnel $U=U_\delta$ by
\be\label{defofU}
U=U_\delta= ((B(p_1,\delta/2) \setminus B(p_1,\delta_0)) \disjointunion (2U_{\delta_0,\gamma}') \disjointunion ((B(p_2,\delta/2) \setminus B(p_2,\delta_0)),
\ee 
where $0<\delta_0<\delta/2$ and $U_{\gamma}'=U_{\delta_0,\gamma}'$ is a modified Gromov-Lawson tunnel, see Figure~\ref{Fig.SY-GL-tunnel}.
%

%\begin{figure}[h]
%	\begin{center%}
%		\includeg%raphics[scale=.15]{tunnelsUprime.png}  % FIGURE
%	\end{center}
%	\caption{The tunnels $U'$.}\label{F:tunnelsU'}
%\end{figure}
%
The boundary of $2U_{\gamma}'$ is isometric to a collar of $B(p_1,\delta_0) \disjointunion B(p_2,\delta_0)$ so we may smoothly attach it to form (\ref{defofU}). \\

% STEP 2 outline
\noindent 
\textbf{Step 2: Construction of the curve $\gamma$, Part 1: $C^{1}$.}
In this step, we construct a $C^1$, and piecewise $C^\infty$, curve $\gamma$. The construction is based on the  bending argument of Gromov and Lawson and uses the fundamental theorem of plane curves i.e.\ the fact that a smooth curve parametrized by arclength is uniquely determined by its curvature, the initial point and the initial tangent vector.
Care must be taken to ensure that the induced metric on $U_{\gamma}'$ maintains positive scalar curvature and that the legth of $\gamma$ is controlled to yield diameter and volume estimates of Lemma \ref{tunnellemma}. 
This step is quite technical and forms the heart of the proof.\\
%

% STEP 3 outline - NEW! C^{1} to C^{\infty}
\noindent 
\textbf{Step 3: Construction of the curve $\gamma$, Part 2: from $C^{1}$ to $C^{\infty}$.}\\ 
In this step we show how to modify the curve constructed in Step 2 to obtain a smooth curve $\bar{\gamma}$ while maintaining all the required features. The modification is elementary and,
once it is completed, we rename $\bar{\gamma}$ back to $\gamma$.\\

% STEP 4 outline 
\noindent 
\textbf{Step 4: Diameter estimates (\ref{TL-diameterU}), (\ref{TL-tunnellengthtozero}) and
volume estimates (\ref{TL-volumeestU}), (\ref{TL-volumeestN}).}\\ 
This is very straightforward since the previous steps give an estimate of the length of the tunnel.

We remark here that the choice of $\delta_0$ is used only to insure that the tunnel $U'$ (see Figure~\ref{Fig.SY-GL-tunnel}) has sufficiently small volume.

%\begin{proof}[Proof of Lemma~\ref{tunnellemma}] % ADDED! Doesn't seem to show up in document ...

% STEP 1: SET-UP AND DEFINITION OF U
\subsection{Step 1 of the Proof.} 
We now set-up our notation further, describe $U$ explicitly in terms of a special curve $\gamma$, and state the important curvature formulas needed in later steps. The  construction of $\gamma$ is done in the next two sub-sections (Steps 2 and 3). 

As mentioned in subsection~\ref{subsection:outline}, because we assume that $B_1$ and $B_2$ have constant sectional curvature $K$ we may work directly in Euclidean space $\R^4$ with coordinates
$(x_0,x_1,x_2,x_3)$ and its standard metric. Let $\gamma(s)$ be a curve in the $(x_0,x_1)$-plane, parametrized by arc-length, written as $\gamma(s)=(x_0(s),x_1(s))$. This curve specifies a hypersurface in $\R^4$ (by rotating $\gamma$ about the $x_0$-axis), 
\be\label{defofU'} 
U' = U_{\gamma}' = \{\, (x_0,x_1,x_2,x_3 \in \R^4 \mid x_0=x_0(s),\,x_1^2+x_2^2+x_3^2 = x_1(s)^2\, \},
\ee 
which we endow with the induced metric. 
Our curve $\gamma$ will always lie in the first quadrant of $(x_0,x_1)$-plane and will be parametrized so that $x_0(s)$ will be increasing. 
We denote by $\theta(s)$ the angle between the horizontal direction and the upward normal vector, and by $\varphi(s)$ the angle between the horizontal direction and the tangent
vector to $\gamma$. 
\begin{figure}[h]
	\begin{center}
	\includegraphics[scale=.2]{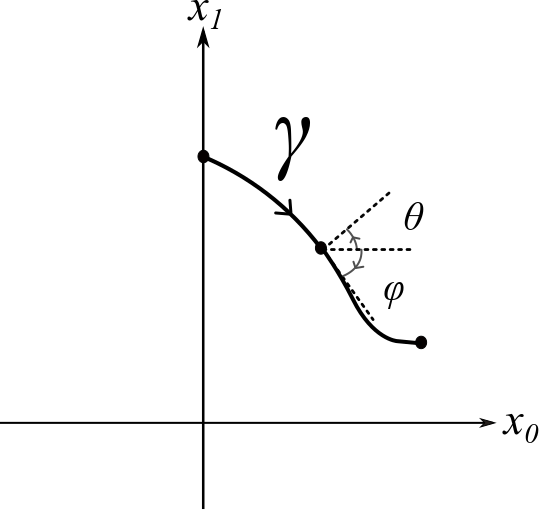} % FIGURE
	\end{center}
	\caption{The curve $\gamma$. }\label{F:gammak}
\end{figure}

We remark that the two angle functions are related by
\be 
\theta(s)=\varphi(s) + \frac{\pi}{2}, \label{theta-phi}
\ee 
see Figure~\ref{F:gammak}. In particular, $\varphi\in (-\pi/2,0]$.

Denote by $k(s)$ the geodesic curvature of $\gamma$. It is a signed quantity so that $\gamma$
bends away from the horizontal axis if $k(s)>0$ and toward the $x_0$-axis when $k(s)<0$. If $\gamma(s_0)=(c,d)$ and $\varphi_0=
\varphi(s_0)$ then (cf.~Theorem 6.7, \cite{Gray-CurveSurfBook}) the function $k(s)$ determines $\gamma$ by the formulae
\be\label{curvaturegivescurve1} 
\varphi(s)=\varphi_0+\int_{s_0}^s k(u)\, du
\ee
and
\be\label{curvaturegivescurve2} 
\gamma(s) = \left(
c+\int_{s_0}^s \cos(\varphi(u))\, du,\, d+\int_{s_0}^s \sin(\varphi(u))\, du
\right).
\ee 

Our aim is to define a function $k(s)$ so that the resulting threefold of revolution $U'$ has
positive scalar curvature. The formula on page 226 of \cite{Gromov-Lawson-tunnels} for $n=3$ gives a relation between the two curvatures. Namely
	\be\label{tunnellemma-neck-scalarcurv}
	\Scal_{U'}(s) = \frac{2\, \sin \theta(s)}{x_1(s)} \left[ \frac{\sin \theta(s)}{x_1(s)} -2 k(s) \right] % SINCE USING THIS FORMULA NEED TO USE THETA AS GL DEFINE IT!!!!!!
	\ee
	where $\Scal_{U'}(s)$ is the scalar curvature of the induced metric on $U'$ and $k$ is the geodesic curvature of $\gamma$. 
	In particular, the formula holds if $\gamma$ is the intersection of the 3-sphere around the origin with the $(x_0,x_1)$-plane in which case $k$ is a negative constant.

We begin defining our curve $\gamma(s)$ so that $\gamma(0)$  corresponds to a point on 
$\partial B(p_1,\delta_0)$ and $\gamma(s)$, for small values of $s\in [0,s_0]$, parametrizes the intersection of $B(p_1,\delta_0)$ with the $(x_0,x_1)$-plane. In particular, for small $s$,
$k(s)\equiv -\sqrt{K}$. We choose $s_0 =\delta_0/2$ and then extend (in Step 2, Subsection \ref{subsection:step2}) the function $k(s)$ to a suitable step function on a longer 
interval $[0,L]$ so that the resulting curve $\gamma(s)$ has the following properties.\\
	
% new
% Roman numerals in enumerate, so can reference as Roman numberal too
\renewcommand\labelenumi{(\Roman{enumi})}
\renewcommand\theenumi\labelenumi
\begin{enumerate}
	\item\label{propofgamma1} 
	%Write $\gamma(s)=(x_0(s),x_1(s))$, that is, $x_i(s)$ will denote the coordinate functions of $\gamma$. 
	The graph of $\gamma$ lies strictly in the first quadrant, beginning at 
	$p_I=\gamma(0)=(0,\cos(-\pi/2+\delta_0)/\sqrt{K})$ 
	and ending at $p_F=\gamma(L)$ 
	with $x_0(L)>0$, $x_1(L)>0$, where $L$ is the length of the curve. Moreover, a point of $\gamma$ moves to the right when $s$ increases.

	\item\label{propofgamma2}  
	Let $\theta(s)$ be the angle between the upward pointing normal to $\gamma$ and the $x_0$-axis. The curve $\gamma$ ends at $p_F$ with $\theta(L)=\pi/2$  and has $\theta=\pi/2$
	(so that it is a horizontal line segment) for an arbitrarily small interval $(L',L]$ (where $L'<L$). 
	
	\item\label{propofgamma3} 	
	
	The curve $\gamma$ has constant curvature $-\sqrt{K}$ near 0 so that the boundary of $U$ has a neighborhood that is isometric to a collar  of $B_1 \cup B_2$.

	\item\label{propofgamma4}  
	The curvature function $k(s)$ satisfies
	\be\label{tunnellemma-scalarcurv-conditiononk}
	k(s) < \frac{\sin(\theta(s))}{2x_{1}(s)} \qquad s \in [0,L],
	\ee
	so that the expression on the right-hand side of (\ref{tunnellemma-neck-scalarcurv}) is positive for all $s \in [0,L]$.  We remark here
	that in certain stages of the construction $k(s)$ will have discontinuities so that $\Scal_{U'}(s)$ is not defined but this will
	cause no difficulties.
	
	\item\label{propofgamma5}
	The length of $\gamma$, $L$, is $O(\delta_{0})$.
\end{enumerate}

Due to properties \ref{propofgamma1} and \ref{propofgamma2} of $\gamma$ above, we may smoothly attach  two copies of $U'$ along their common boundary at $s=L$ to define $2U'=U_{\gamma}' \disjointunion U_{\gamma}'$ 
and then, using property \ref{propofgamma3}, attach $2U'$ to form $U$ as in (\ref{defofU}). 

In the next step, we construct a piecewise $C^1$ curve $\gamma$ in the $(x_0,x_1)$-plane which satisfies properties \ref{propofgamma1} through \ref{propofgamma5}. Then, in Step 3, we modify the construction once more to produce a smooth curve, $\bar{\gamma}$, with these same properties.

% STEP 2: CONSTRUCTION OF THE CURVE GAMMA, C^{1}
\subsection{Step 2 of the Proof: Construction of $\gamma$, Part 1: $C^{1}$.}\label{subsection:step2}
%
% begin new construction 

As above, let $s_0=\delta_0/2$  and let $q_{0}=(a_{0},b_{0})$ be the coordinates of the point $\gamma(s_{0})$ that is already defined. By choosing $\delta_0$ sufficiently small we can assume that the tangent vector to $\gamma$ at $s=s_{0}$ is nearly vertical and is pointing downward at $s=s_{0}$.  We also have $k(s)\equiv -\sqrt{K}$ on $[0,s_0]$.\\

%\ref{propofgamma1} % first Quad
%, \ref{propofgamma2}, \ref{propofgamma3} % angle increases and ends at theta=pi/2
%\ref{propofgamma4} and \ref{propofgamma5}, % scal>0
%\ref{propofgamma6} % L to 0

% inductive extensions 
We will use a finite induction  to define a sequence of extensions of $\gamma$ over intervals $[s_{i},s_{i+1}]$, with $s_{i}<s_{i+1}$ for a finite number of steps $0 \le i \le n$, where $n=n(\delta_{0})$ is the number of steps required such that properties \ref{propofgamma1},  \ref{propofgamma3},  \ref{propofgamma4}, and \ref{propofgamma5} all hold at each extension. We denote by $(a_{i},b_{i})$ the coordinates of the point $\gamma(s_{i})$ for $0 \le i \le n$. 

Let us first choose the curvature function $k(s)$ of $\gamma(s)$ on the first extended interval $[s_{0},s_{1}]$. Observe that equation (\ref{tunnellemma-scalarcurv-conditiononk}) limits the amount of positive curvature allowed for $k(s)$. 
In fact, we choose $k(s)$ to be the constant $k_{1}>0$ over the interval $[s_{0},s_{1}]$ based only the initial data at $s_{0}$
\be\label{k1choice} % CHOICE K1
	k_{1} = \frac{\sin(\theta(s_{0}))}{4b_{0}}>0,
\ee
where $\theta(s_{0})=\frac{\pi}{2}+\varphi(s_{0}) = \delta_{0}-\sqrt{K}s_{0}>0$ and $b_{0}=x_{1}(s_{0})$.  
Note that constant positive curvature means that $\gamma(s)$ moves along the arc of a circle of curvature $1/\sqrt{k_{1}}$ bending away from the origin. 

We verify that property \ref{propofgamma4} holds with our choice of $k_{1}$ in (\ref{k1choice}). 
From (\ref{curvaturegivescurve1}), we see that $\varphi(s)$ is an increasing function with range in the interval $(-\pi/2,0)$, hence $\theta(s)$ is also increasing by (\ref{theta-phi}). Moreover, from (\ref{curvaturegivescurve1}) and (\ref{curvaturegivescurve2}), we see that the $x_{1}$-coordinate function is decreasing on the interval $(s_{0}, s_{1})$ since $x_{1}'(s)=\sin(\varphi(s))<0$.  Thus, the expression on the right-hand side of  (\ref{tunnellemma-scalarcurv-conditiononk}), $\sin(\theta(s))/(2x_{1}(s))$, is an increasing function on $(s_{0},s_{1})$ so that
\be\label{RHScurvatureconditionincreasing}
\frac{\sin(\theta(s_{0}))}{2x_{1}(s_{0})} \le \frac{\sin(\theta(s))}{2x_{1}(s)} 
\qquad s \in [s_{0},s_{1}].
\ee
Since  $k(s)\equiv k_{1}$ is constant it follows that the property \ref{propofgamma4}  holds for $s \in [s_{0},s_{1}]$. 

Next, we choose the length of the extension $\Delta s_{1} = s_{1} -s_{0}$, so that properties \ref{propofgamma1} and \ref{propofgamma5} hold. This is achieved by setting
\be\label{Deltas1choice} % CHOICE DELTA S1 
	\Delta s_{1} = \frac{b_{0}}{2}>0 
\ee 
Observe that $x_0(s)$ is increasing since $x_{0}'(s)=\cos(\varphi(s))>0$ as $\varphi \in (-\pi/2,0)$. 
%
%It follows that property \ref{propofgamma1} holds for $s \in [s_{0},s_{1}]$. 

% L to 0
Clearly we have
\be\label{b0tozeroasdeltato0}
	b_{0} < \delta_{0}
\ee
since $b_0$ is the vertical distance of $\gamma(s_0)$ to the $x_0$-axis which is less than the
distance along the sphere.

% change in angle
Of course, we do not achieve a final angle of $\pi/2$ of the normal at $s_{1}$ and gain only a small but definite increase in the angle. %and so we iterate this procedure. 
The change in angle of the normal with the $x_{0}$-axis is
$$
\Delta \theta_{1} =\theta(s_{1}) - \theta(s_{0}) 
	= \int_{s_{0}}^{s_{1}} k(s)\, ds = k_{1} \cdot \Delta s_{1} 
	= \frac{\sin(\theta(s_{0}))}{8} >0
	$$
	by (\ref{k1choice}) and (\ref{Deltas1choice}).

% inductive step
With $\gamma$ extended over the first interval $[s_{0},s_{1}]$, we now inductively define further extensions.
Assume that $\Delta s_{j}$, $s_{j}$ and $k_{j}$ have been chosen for $j=1,2,\ldots,(i-1)$, and $\gamma$ extended on the intervals $[s_{j},s_{j+1}]$, we then define 
\be\label{choosesiki}
	\Delta s_{i} = \frac{b_{i-1}}{2}, \qquad 
	s_{i} = s_{i-1} + \Delta s_{i}%s_{0} + \sum_{j=1}^{i-1} \Delta s_{j}
	\qquad \text{and} \qquad
	k_{i} = \frac{\sin(\theta(s_{i-1}))}{4b_{i-1}},
\ee
where $\gamma(s_{i})=(a_{i},b_{i})$. In what follows we will also write $\theta_j$ and $\varphi_j$ for $\theta(s_j)$ and $\varphi(s_j)$ respectively. We remark that 
$b_{i+1} < b_i$ by (\ref{curvaturegivescurve2}) since the angle $\varphi$ is negative and 
that $k_{i+1} > k_i$ since the ratio $\frac{\sin(\theta(s))}{x_{1}(s)}$ is increasing. 
Observe that properties \ref{propofgamma1}, \ref{propofgamma4}, and \ref{propofgamma5} of $\gamma$ hold on $[s_{i-1},s_{i}]$ for all $i$ by our choices in (\ref{choosesiki}) by arguments analogous to those given for the first extension of $\gamma$ on $[s_{0},s_{1}]$. 
%%
%Unfortunately, this construction is not enough to ensure property \ref{propofgamma6} as we now discuss.

% definite change in angle at each extension
We observe that we gain a definite amount of angle $\theta$ with each extension since, by (\ref{choosesiki}), for each $j \in \{1,2,\ldots, i\}$,
\begin{align}
\Delta \theta_{j} 
	=\theta(s_{j}) - \theta(s_{j-1}) 
	&= \int_{s_{j-1}}^{s_{i}} k(s)\, ds = k_{j} \cdot \Delta s_{j} 
	= \frac{\sin(\theta(s_{j-1}))}{8} \qquad  \notag\\
		&\ge \frac{\sin(\theta(s_{0}))}{8}, \label{thetaest}
\end{align}
because $\theta(s_{j-1}) \ge \theta(s_{0})$ and the the values of $\theta$ are in the range $(0,\pi/2)$ so that the sine is an increasing function. We stop the construction when
$\theta(s)$ reaches the value $\pi/2$.
Thus the total change in the angle $\theta$ over the interval $[0,s_{i}]$ is bounded from below by
\be\label{changeinthetagetsbig} 
\Delta \theta =  
	\sum_{j=1}^{i} \Delta \theta_{j} \ge  i \cdot \frac{\sin(\theta_{0})}{8}.
\ee

To prove property \ref{propofgamma5}, that the length of $\gamma$ is on the order of $\delta_{0}$, we  need the sequence of $b_{i}$'s to be summable and will want to compare it to the geometric progression. The difficulty here is that, since our curve is bending more and more upwards,
the ratios $b_i/b_{i-1}$ increase. For this reason we stop our induction when $\theta$ reaches
the value of $\pi/4$. It will turn out that once this value is reached, we can complete the
construction of $k(s)$ by a single extension albeit with $\Delta s$ not given by (\ref{choosesiki}).

Thus, define  $n=n(\delta_{0})$ to be the first positive integer with 
\be\label{choiceofn}
	\frac{\pi}{4} \le \theta_{n} 
\ee
which exists by (\ref{changeinthetagetsbig}).
Moreover,
if $\theta_{n}>\pi/4$ we re-define $s_{n}$ to be the exact value in $(s_{n-1},\infty)$ such that $\theta(s_{n})=\pi/4$. Thus, for the modified value of $s_n$
\be\label{snthetan=pi/4}
	\theta_{n} = \theta(s_n) = \frac{\pi}{4}.
\ee

% Length to 0
%An essential ingredient for showing property \ref{propofgamma6} is that the sequence of lengths of the extensions is summable. This is precisely what the Claim below allows.\\

The following Lemma gives the desired comparison.

%claim 
\bl\label{claim}
There exists a universal constant $C\in (0,1)$, independent of $\delta_{0}$ and $K$, such that 
for all $i \le n$
$$
	b_{i} \le C \cdot b_{i-1}, 
$$
where $n=n(\delta_0)$ is as above.
\el

% proof that length goes to zero
The Lemma, to be proven shortly below, implies that the length of the curve $\gamma$ on the entire interval $[0,s_{n}]$ is no larger than a constant (independent of $\delta_{0}$) times $\delta_{0}$. Namely,
\be
	L(\gamma([0,s_{n}])) = s_{n} = \sum_{j=1}^{n} \Delta s_{j}.
\ee
Thus, from (\ref{choosesiki}) and Lemma (\ref{claim}), we have
\begin{equation}
	\sum_{j=1}^{n} \Delta s_{j} = \sum_{j=1}^{n} \frac{b_{j-1}}{2} 
	\le \frac{b_{0}}{2} \left( \sum_{j=1}^{n-1} C^{j} \right) \leq C_1 \delta_0
	 \label{lengthsn}
\end{equation}
by the lemma and (\ref{b0tozeroasdeltato0}).
So, $L(\gamma([0,s_{n}])) \le C' b_{0}$ with $C'= \frac{1}{2-2C}$ which is independent of $\delta_{0}$ since $C$ is. This proves that $L(\gamma([0,s_{n}])) = O(\delta_0)$. \\

% proof of claim
%We now prove Lemma \ref{claim}.
%
\begin{proof}[Proof of Lemma \ref{claim}]
Let $1 \le i \le n$. We compute explicitely using (\ref{curvaturegivescurve1}), (\ref{curvaturegivescurve2}) and (\ref{choosesiki}),
\be
\varphi(s_{i}) = 
 \varphi(s_{i-1}) + k_{i}\cdot \Delta s_{i} 
=\varphi(s_{i-1})+\frac{\sin(\theta_{i-1})}{8}\label{varphisi}
\ee
and
\begin{align*}
b_{i} &= x_{1}(s_{i}) \\%= x_{1}(s_{0}+\Delta s_{1}) \\
&= b_{i-1} + \int_{s_{i-1}}^{s_{i}}  \sin(\varphi(s_{i-1})+k_{i}(u-s_{i-1}))\, du\\
&= b_{i-1} - \frac{1}{k_{i}} \left( \cos(\varphi(s_{i})) - \cos(\varphi(s_{i-1})) \right)\\
&=  b_{i-1} - \frac{4b_{i-1}}{\sin(\theta(s_{i-1}))} \left( \cos \left(\varphi(s_{i-1})+\frac{\sin(\theta_{i-1})}{8} \right) - \cos(\varphi(s_{i-1})) \right). 
\end{align*}
Thus,
\begin{align*}
\frac{b_{i}}{b_{i-1}} &= 1 - \frac{4}{\sin(\theta(s_{i-1}))} \left( \cos \left( \varphi(s_{i-1})+\frac{\sin(\theta_{i-1})}{8} \right) - \cos(\varphi(s_{i-1})) \right).
\end{align*}
Therefore, by the Mean Value Theorem, there exists $\mu_{i} \in (\varphi(s_{i-1}),\varphi(s_{i-1})+\sin(\theta(s_{i-1}))/8)$ such that
\begin{align*}
\frac{b_{i}}{b_{i-1}} &= 1 - \frac{4}{\sin(\theta(s_{i-1}))} (-\sin(\mu_{i})) \cdot \frac{\sin(\theta(s_{i-1}))}{8} = 1 + \frac{\sin(\mu_{i})}{2}. 
\end{align*}
To complete the proof of the claim, we seek a constant $0<C<1$, independent of $\delta_{0}$, such that
\begin{align}\label{equivalentclaim}
1 + \frac{\sin(\mu_{i})}{2} < C < 1.
\end{align}
Recall that the angle function $\varphi$ takes negative values throughout. 

We claim that the choice 
\be\label{choiceofC}
C=1+\frac{1}{4} \sin \left( -\frac{\pi}{4}+\frac{\cos(-\frac{\pi}{4})}{8} \right) 
\approx
0.8395
\ee
will satisfy our requirement. 

This follows from the fact that the sine is an increasing function on the interval $(\varphi(s_{i-1}),\varphi(s_{i-1})+\sin(\theta(s_{i-1}))/8)$ and the fact that both the angles $\varphi_{i}$ and $\theta_{i}$ are increasing, so
\begin{align*}
1+\frac{\sin(\mu_{i})}{2} &\le 1+\frac{1}{2}\sin \left( \varphi(s_{i-1})+\frac{\sin(\theta(s_{i-1}))}{8} \right)\\
&\le 1+ \frac{1}{2} \sin \left( \varphi(s_{n})+\frac{\cos(\varphi(s_{n}))}{8} \right).
\end{align*}
% note on need choice of n
By our choice of $s_n$, $\theta(s_{n})=\pi/4$ from (\ref{snthetan=pi/4}) and $\varphi(s_{n})=-\pi/4$ so that
\begin{align*}
1+\frac{\sin(\mu_{i})}{2} &\le
1+ \frac{1}{2} \sin \left( -\frac{\pi}{4}+\frac{\cos \left(-\frac{\pi}{4} \right)}{8} \right)\\
&< 1+ \frac{1}{4} \sin \left( -\frac{\pi}{4}+\frac{\cos \left(-\frac{\pi}{4} \right)}{8} \right)\\
&= C <1.
\end{align*}
This finishes the proof of the Lemma.
\end{proof}
At this stage of the construction, $\gamma$ has angle $\theta=\pi/4$ at the endpoint $s_{n}$.   
We make one additional extension of our step function.\\

%FINAL BEND

We now define $s_{n+1}>s_{n}$ and $k_{n+1}>0$ as follows.  

By (\ref{curvaturegivescurve1}) $\varphi(s)$ in $[s_{n},s_{n+1}]$ will be given by
\be\label{varphifinal1}
	\varphi(s) = \varphi_{n} + \int_{s_{n}}^{s} k(u)\, du = \varphi_{n} + k_{n+1}(s-s_{n}).
\ee 
Let  $s_{n+1}$ be determined by $k_{n+1}$ as the first value such that $\varphi(s_{n+1})=0$ (equivalently $\theta(s_{n+1})=\pi/2$). Then 
\be\label{varphi=0}
	0= \varphi(s_{n+1})  = \varphi_{n} + k_{n+1}(s_{n+1}-s_{n})
\ee
so that 
\be\label{sn+1} 
	s_{n+1} = s_{n} - \frac{\varphi_{n}}{k_{n+1}}.
\ee
We require in addition that $b(s_{n+1})>0$ (that is, $\gamma$ remains above the $x_{0}$-axis).  Using (\ref{sn+1}) and (\ref{curvaturegivescurve2}), we obtain
\begin{align}
	b(s_{n+1}) &= b_{n} + \int_{s_{n}}^{s_{n+1}} \sin(\varphi(s))\, ds
		%= b_{n} - \left. \frac{\cos(\varphi(s))}{k_{n+1}} \right|_{s_{n}}^{s_{max}}
		= b_{n} - \frac{\cos(\varphi(s_{n+1})) - \cos(\varphi(s_{n})) }{k_{n+1}} \notag\\
		&= b_{n} - \frac{1 - \cos(\varphi(s_{n})) }{k_{n+1}} 
\end{align}
so that $b(s_{n+1})>0$ is equivalent to 
\[
	b_{n} - \frac{ 1 - \cos(\varphi(s_{n})) }{k_{n+1}} > 0
\]
or
\be\label{kbelow}
	k_{n+1}\cdot b_{n} > 1 - \cos(\varphi(s_{n})).
\ee

On the other hand, $k_{n+1}$ has to be bounded from above in order to guarantee (\ref{tunnellemma-scalarcurv-conditiononk}). Therefore, we require that
\[
	k_{n+1} < \frac{\sin(\theta(s_{n}))}{2b_{n}},
\]
or
\be\label{kabove}
	k_{n+1}\cdot b_{n} < \frac{\sin(\theta(s_{n}))}{2}.
\ee
Combining (\ref{kbelow}) and (\ref{kabove}) gives  conditions for $k_{n+1}$
\be\label{conditiononkn+1-1}
	1 - \cos(\varphi(s_{n})) < k_{n+1}\cdot b_{n} < \frac{\sin(\theta(s_{n}))}{2}.
\ee
Since $\sin(\theta(s))=\cos(\varphi(s))$, (\ref{conditiononkn+1-1}) is equivalent to
\be\label{conditiononkn+1-2}
	1 - \cos(\varphi(s_{n})) < k_{n+1}\cdot b_{n} < \frac{\cos(\varphi(s_{n}))}{2}.
\ee

Now, recall that $s_{n}$ was chosen in (\ref{snthetan=pi/4}) so that $\varphi(s_{n})=-\pi/4$ so 
\[
1-\cos(\varphi(s_{n}))=\frac{2-\sqrt{2}}{2} < \frac{\cos(\varphi(s_{n}))}{2} = \frac{\sqrt{2}}{4}.
\]
Now, choose arbitrarily any $\alpha$,  satisfying
\be\label{defofalpha}
\frac{2-\sqrt{2}}{2} < \alpha < \frac{\sqrt{2}}{4}
\ee
and define $k_{n+1}$ by 
\be\label{choicekn+1}
k_{n+1}=\alpha/b_{n}.
\ee 
With this choice (\ref{conditiononkn+1-2}), and therefore, (\ref{kbelow}) and (\ref{kabove}) hold.
%and notice that (\ref{conditiononkn+1-2}) is indeed fulfilled. \\

%\vspace{-10pt}
\begin{figure}[h]
	\begin{center}
	\includegraphics[scale=.25]{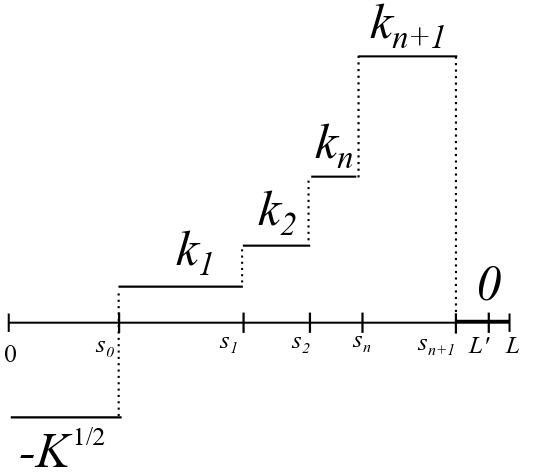}
	\end{center}
	\caption{Graph of the curvature, $k(s)$, with ``full bend'' as a step function.}\label{F:finalkC1} 
\end{figure} 

% one last extension to have tiny bit of straight line segment to glue
To ensure property \ref{propofgamma2}, we  choose $L>s_{n+1}$ so that $L-s_{n+1}$ is arbitrarily small. We extend $\gamma$ to the interval $[s_{n+1},L]$ where $\gamma$ is a straight horizontal line on $[s_{n+1},L]$ by choosing $k(s)=0$ there. To check that the length of the
curve we constructed is $O(\gamma_0)$ we observe that 
\be\label{sn+1to0}
s_{n+1} = s_{n} - \varphi_{n}/k_{n+1} = s_{n} + \frac{\pi}{4\alpha}b_{n} 
	\le s_{n} + \frac{\pi}{4\alpha}b_{0} = O(\delta_{0})
\ee
by (\ref{b0tozeroasdeltato0}), (\ref{lengthsn}) and (\ref{sn+1to0}).

We note that the choice of $L$ is arbitrary. It will be made explicit in the next step when we construct the curve $\bar{\gamma}$, the $C^{\infty}$ version of $\gamma$.\\

This completes the construction of the continuously differentiable curve $\gamma$ defined on the interval $[0,L]$ satisfying properties \ref{propofgamma1} through \ref{propofgamma5}.

% STEP 3: CONSTRUCTION OF THE CURVE GAMMA, PART 2: SMOOTHING
\subsection{Step 3 of the Proof: Construction of $\gamma$, Part 2: from $C^{1}$ to $C^{\infty}$.}\label{subsection:step3}
In this section, barred quantities will refer to the $C^{\infty}$ curve $\bar{\gamma}(s)$ to be constructed in this step and all the other quantities related to the construction (for example, $\bar{\theta}$, $\bar{\varphi}$, $\bar{k}(s)$, etc.). Unbarred quantities will refer to the $C^{1}$ curve constructed in the previous step.% (for example, $\gamma(s)$ and $k(s)$, etc.)

The general plan is to replace $k(s)$ as chosen in Step 2 with a smooth version $\bar{k}(s)$ as depicted in Figure~\ref{F:finalksmooth}, which will then define $\bar{\gamma}$ by the
formulae (\ref{curvaturegivescurve1}) and (\ref{curvaturegivescurve2}). Set $k_0=-K^{1/2}$ and modify $k(s)$ on $[s_i,s_{i+1}]$ for $i=0,1,2,\ldots ,n$ so that
 the graph of $\bar{k}(s)$ will connect to the constant function equal to $k_{i}$ smoothly at $s_i$, will rise steeply to the value $k_{i+1}$ in a very short interval $[s_i,s_i+\alpha]$ and will connect smoothly with constant function equal to $k_{i+1}$ in $[s_i+\alpha, s_{i+1}]$. For each $i=0,1,2,\ldots n$, $\bar{k}|[s_i,s_{i+1}]$ can be constructed as follows.  
Choose and fix a $C^\infty$ function $g(s)$ which is identically 0 for $s<0$, identically 1 for $s>1$, and strictly increasing on $[0,1]$. Then $\bar{k}| [s_i,s_{i+1}]$ is constructed by appropriate rescaling and translations of the graph of $g(s)$
in both vertical and horizontal directions. The values of $k_i$ and $k_{i+1}$ determine the transformations along the
vertical axis but rescaling of the independent variable remains a free parameter $\alpha$ to be
set sufficiently small later.
We will use the same value of $\alpha$ for every $i=1,2,\ldots n$.

\begin{figure}[h]
	\begin{center}
	\includegraphics[scale=.22]{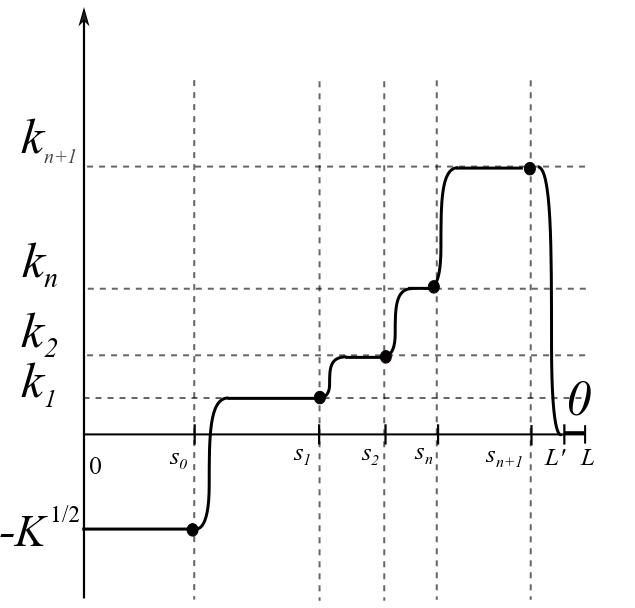}
	\end{center}
	\caption{Graph of the smooth curvature $\bar{k}(s)$ with ``full bend.''}\label{F:finalksmooth} 
\end{figure} 

Since $$\Delta \bar{\theta}=\int_0^{s_{n+1}}\bar{k}\,ds \leq \int_0^{s_{n+1}}{k}\,ds = \Delta \theta ,$$
we loose a small amount of "bend" so that $\bar\theta(s_{n+1}) < \frac{\pi}{2}$ by a very small amount controlled by $\alpha$. We compensate for this by one final extension of $\bar{k}$ to
an interval $[s_{n+1},L]$ with $L=s_{n+1} + 2\beta$. We choose $\bar{k}$ so that it connects
smoothly with $k_{n+1}$ at $s_{n+1}$, drops smoothly to zero over $[s_{n+1}, s_{n+1}+\beta]$
and continues identically zero on $[s_{n+1}+\beta, s_{n+1}+2\beta]$. $\beta$ and
$\bar{k}$ are chosen so that $$
\int_{s_{n+1}}^{s_{n+1}+\beta} \bar{k}(s)\, ds = \frac{\pi}{2} - \bar{\theta}(s_{n+1}).$$
This ensures that $\bar{\theta}=\frac{\pi}{2}$ in the interval $[s_{n+1}+\beta, s_{n+1}+2\beta]$. This final extension 
is constructed as the preceding ones except that we have to use the reflection $s \mapsto -s$ before rescaling and
translating the original fuction $g$. We note
that $\beta=O(\alpha)$ is determined by the choice of $\alpha$ and the requirement that $\bar{\theta}(L)=\frac{\pi}{2}$. We also observe that as $\alpha$ tends to zero, the functions $\bar{\varphi}$, $\bar{\theta}$,
$\bar{x_0}$, and $\bar{x_1}$ will converge uniformly on $[0,L]$ to $\varphi$, $\theta$, $x_0$, and $x_1$ respectively as follows from (\ref{curvaturegivescurve1}) and (\ref{curvaturegivescurve2}).

We now check that the properties \ref{propofgamma1} through \ref{propofgamma5} on page \pageref{propofgamma1} hold for the curve $\bar{\gamma}$ for sufficiently small choice of $\alpha$. Only \ref{propofgamma4} and \ref{propofgamma5} need a verification. \ref{propofgamma5} follows since $L = s_{n+1} + 2\beta = O(\delta_0) + O(\alpha)$. To prove \ref{propofgamma4} we use the uniform convergence on $[0,s_{n+1}]$ as $\alpha$ approaches 0 of
$\frac{\sin
\bar{\theta}(s)}{2\bar{x}_1(s)}$ to $\frac{\sin \theta(s)}{2x_1(s)}$. More precisely, on $[s_i,s_{i+1}]$,
$$
\frac{\sin \bar{\theta}(s)}{2\bar{x}_1(s)} -\bar{k}(s) = \left ( \frac{\sin \bar{\theta}(s)}{2\bar{x}_1(s)} - k_{i+1}\right )
+\left ( k_{i+1} -\bar{k}(s)\right ).
$$
For sufficiently small $\alpha$, the first term on the right becomes positive by the property \ref{propofgamma4}
for the curve $\gamma$ while the second term is nonnegative by construction (cf.\ Figure \ref{F:finalkC1}). Finally,
in the last interval $[s_{n+1},L]$ the ratio $\frac{\sin\bar{\theta} (s)}{2\bar{x}_1(s)}$ is nondecreasing so that
$$
\frac{\sin\bar{\theta}(s)}{2\bar{x}_1(s)}\geq \frac{\sin\bar{\theta}(s_{n+1})}{2\bar{x}_1(s_{n+1})} > k_{n+1}
$$
since the last inequality was verified for $s=s_{n+1}$ already. Property \ref{propofgamma4} follows since
$k_{n+1} > \bar{k}(s)$ in $[s_{n+1},L]$. This finishes the construction of $\bar{\gamma}$.

%%%%%%%%%%%%%%%%%%%%

% STEP 4: DIAMETER ESTIMATES
\subsection{Step 4 of the Proof: Diameter and volume estimates of Lemma \ref{tunnellemma}}
\label{subsection:step4}
Given the definition of $U$ in (\ref{defofU}), the diameter of $U$ is estimated by
$$\label{diameterU0}
\diam(U)
\le \pi \delta+\delta+ 2L = O(\delta)+O(\delta_0)=O(\delta).
%= O(\delta).
%2\pi(\delta/2)+2(\delta/2-\delta_0) + 2L(\gamma(\delta)),
$$
To estimate the volume of $U'$, note that the intersection of $U'$ with the hyperplane $x_0=x_0(s)=c$ for $0<s<L$
is a sphere of two dimensions and of radius $x_1(s)< \delta_0$. It follows by Fubini's theorem that
$\vol(U')=O(\delta_0^3)$. To prove (\ref{TL-volumeestU}) recall that $U$ is obtained from the union of
two disjoint balls of radius $\delta$ by removing balls of radius $\delta_0$ and attaching $U'$ along the
common boundary (cf.\ Figure \ref{Fig.SY-GL-tunnel}). Since the  volumes of the removed balls and of the added tunnel
are $O(\delta_0^3)$, the estimate (\ref{TL-volumeestU}) follows by choosing $\delta_0$ sufficiently small depending on $\epsilon$. The estimate (\ref{TL-volumeestN}) is proved in the same way. The proof of Lemma \ref{tunnellemma} is now complete.
\end{proof}

\bibliographystyle{alpha}
\bibliography{basilio}

\begin{thebibliography}{Gro14b}

\bibitem[AGS14]{AGS}
Luigi Ambrosio, Nicola Gigli, and Giuseppe Savar{\'e}.
\newblock Metric measure spaces with {R}iemannian {R}icci curvature bounded
  from below.
\newblock {\em Duke Math. J.}, 163(7):1405--1490, 2014.

\bibitem[AK00]{AK}
Luigi Ambrosio and Bernd Kirchheim.
\newblock Currents in metric spaces.
\newblock {\em Acta Math.}, 185(1):1--80, 2000.

\bibitem[AT04]{Ambrosio-Tilli-AoMS}
Luigi Ambrosio and Paolo Tilli.
\newblock {\em Topics on analysis in metric spaces}, volume~25 of {\em Oxford
  Lecture Series in Mathematics and its Applications}.
\newblock Oxford University Press, Oxford, 2004.

\bibitem[Bam16]{Bamler-16}
Richard Bamler.
\newblock A {R}icci flow proof of a result by {G}romov on lower bounds for
  scalar curvature.
\newblock {\em Mathematical Research Letters}, 23(2):325--337, 2016.

\bibitem[BBI01]{BBI}
Dmitri Burago, Yuri Burago, and Sergei Ivanov.
\newblock {\em A course in metric geometry}, volume~33 of {\em Graduate Studies
  in Mathematics}.
\newblock American Mathematical Society, Providence, RI, 2001.

\bibitem[BI09]{Burago-Ivanov-Area}
Dimitri Burago and Sergei Ivanov.
\newblock Area spaces: First steps, with appendix by nigel higson.
\newblock {\em Geometric and Functional Analysis}, 19(3):662--677, 2009.

\bibitem[BS17]{Basilio-Sormani-1}
Jorge Basilio and Christina Sormani.
\newblock Sequences of three dimensional manifolds with positive scalar
  curvature.
\newblock {\em preprint to appear}, 2017.

\bibitem[CC97]{ChCo-PartI}
Jeff Cheeger and Tobias~H. Colding.
\newblock On the structure of spaces with {R}icci curvature bounded below. {I}.
\newblock {\em J. Differential Geom.}, 46(3):406--480, 1997.

\bibitem[FF60]{FF}
Herbert Federer and Wendell~H. Fleming.
\newblock Normal and integral currents.
\newblock {\em Ann. of Math. (2)}, 72:458--520, 1960.

\bibitem[Fuk87]{Fukaya-1987}
Kenji Fukaya.
\newblock Collapsing of {R}iemannian manifolds and eigenvalues of {L}aplace
  operator.
\newblock {\em Invent. Math.}, 87(3):517--547, 1987.

\bibitem[GL80a]{Gromov-Lawson-tunnels2}
Mikhael Gromov and H.~Blaine Lawson.
\newblock The classification of simply connected manifolds of positive scalar
  curvature.
\newblock {\em Annals of Mathematics}, 111(3):423--434, 1980.

\bibitem[GL80b]{Gromov-Lawson-tunnels}
Mikhael Gromov and H.~Blaine Lawson, Jr.
\newblock Spin and scalar curvature in the presence of a fundamental group.
  {I}.
\newblock {\em Ann. of Math. (2)}, 111(2):209--230, 1980.

\bibitem[Gra98]{Gray-CurveSurfBook}
Alfred Gray.
\newblock {\em Modern differential geometry of curves and surfaces with
  Mathematica}.
\newblock CRC Press, second edition, 1998.

\bibitem[Gro99]{Gromov-metric}
Misha Gromov.
\newblock {\em Metric structures for {R}iemannian and non-{R}iemannian spaces},
  volume 152 of {\em Progress in Mathematics}.
\newblock Birkh\"auser Boston Inc., Boston, MA, 1999.
\newblock Based on the 1981 French original [ MR0682063 (85e:53051)], With
  appendices by M. Katz, P. Pansu and S. Semmes, Translated from the French by
  Sean Michael Bates.

\bibitem[Gro14a]{Gromov-Dirac}
Misha Gromov.
\newblock Dirac and {P}lateau billiards in domains with corners.
\newblock {\em Cent. Eur. J. Math.}, 12(8):1109--1156, 2014.

\bibitem[Gro14b]{Gromov-Plateau}
Misha Gromov.
\newblock Plateau-{S}tein manifolds.
\newblock {\em Cent. Eur. J. Math.}, 12(7):923--951, 2014.

\bibitem[HLS17]{HLS}
Lan-Hsuan Huang, Dan Lee, and Christina Sormani.
\newblock Stability of the positive mass theorem for graphical hypersurfaces of
  {E}uclidean space.
\newblock {\em Journal fur die Riene und Angewandte Mathematik (Crelle's
  Journal)}, 727:269--299, 2017.

\bibitem[LS13]{Lakzian-Sormani}
Sajjad Lakzian and Christina Sormani.
\newblock Smooth convergence away from singular sets.
\newblock {\em Comm. Anal. Geom.}, 21(1):39--104, 2013.

\bibitem[LS14]{LeeSormani1}
Dan~A. Lee and Christina Sormani.
\newblock {S}tability of the positive mass theorem for rotationally symmetric
  riemannian manifolds.
\newblock {\em Journal fur die Riene und Angewandte Mathematik (Crelle's
  Journal)}, 686, 2014.

\bibitem[LS15]{LeFloch-Sormani-1}
Philippe~G. LeFloch and Christina Sormani.
\newblock The nonlinear stability of rotationally symmetric spaces with low
  regularity.
\newblock {\em J. Funct. Anal.}, 268(7):2005--2065, 2015.

\bibitem[LV09]{Lott-Villani-09}
John Lott and C{\'e}dric Villani.
\newblock {R}icci curvature for metric-measure spaces via optimal transport.
\newblock {\em Ann. of Math. (2)}, 169(3):903--991, 2009.

\bibitem[MP15]{Matveev-Portegies}
Rostitslav Matveev and Jacobus Portegies.
\newblock Intrinsic flat and {G}romov-{H}ausdorff convergence of manifolds with
  {R}icci curvature bounded below.
\newblock {\em arXiv:1510.07547}, 2015.

\bibitem[Por15]{Portegies-F-evalue}
Jacobus~W. Portegies.
\newblock Semicontinuity of eigenvalues under intrinsic flat convergence.
\newblock {\em Calc. Var. Partial Differential Equations}, 54(2):1725--1766,
  2015.

\bibitem[RS01]{Rosenberg-Stolz-PSC2001}
Jonathan Rosenberg and Stephen Stolz.
\newblock {M}etrics of positive scalar curvature and connections with surgery.
\newblock In Andrew~Ranicki Sylvain~Cappell and Jonathan Rosenberg, editors,
  {\em {S}urveys on {S}urgery {T}heory}, number 149 in Annals of Mathematics
  Studies 2. Princeton University Press, 2001.

\bibitem[Sor14]{Sormani-AA}
Christina Sormani.
\newblock Intrinsic flat {A}rzela-{A}scoli theorems.
\newblock {\em Communications in Analysis and Geometry}, 27(1), 2019 (on arxiv
  since 2014).

\bibitem[Sor17]{Sormani-scalar}
Christina Sormani.
\newblock Scalar curvature and intrinsic flat convergence.
\newblock In Nicola Gigli, editor, {\em Measure Theory in Non-Smooth Spaces},
  chapter~9, pages 288--338. De Gruyter Press, 2017.

\bibitem[SS17]{Sormani-Stavrov-1}
C~Sormani and I~Stavrov.
\newblock Geometrostatic manifolds of small {ADM} mass.
\newblock {\em arXiv: 1707.03008}, 2017.

\bibitem[Stu06a]{Sturm-curv}
Karl-Theodor Sturm.
\newblock A curvature-dimension condition for metric measure spaces.
\newblock {\em C. R. Math. Acad. Sci. Paris}, 342:197?200, 2006.

\bibitem[Stu06b]{Sturm-06}
Karl-Theodor Sturm.
\newblock On the geometry of metric measure spaces. {I}.
\newblock {\em Acta Math.}, 196(1):65--131, 2006.

\bibitem[SW10]{SW-CVPDE}
Christina Sormani and Stefan Wenger.
\newblock Weak convergence and cancellation, appendix by {R}aanan {S}chul and
  {S}tefan {W}enger.
\newblock {\em Calculus of Variations and Partial Differential Equations},
  38(1-2), 2010.

\bibitem[SW11]{SW-JDG}
Christina Sormani and Stefan Wenger.
\newblock Intrinsic flat convergence of manifolds and other integral current
  spaces.
\newblock {\em Journal of Differential Geometry}, 87, 2011.

\bibitem[SY79a]{Schoen-Yau-tunnels}
R.~Schoen and S.~T. Yau.
\newblock On the structure of manifolds with positive scalar curvature.
\newblock {\em Manuscripta Math.}, 28(1-3):159--183, 1979.

\bibitem[SY79b]{Schoen-Yau-positive-mass}
Richard Schoen and Shing~Tung Yau.
\newblock On the proof of the positive mass conjecture in general relativity.
\newblock {\em Comm. Math. Phys.}, 65(1):45--76, 1979.

\bibitem[Vil09]{Villani-text}
C{\'e}dric Villani.
\newblock {\em Optimal transport}, volume 338 of {\em Grundlehren der
  Mathematischen Wissenschaften [Fundamental Principles of Mathematical
  Sciences]}.
\newblock Springer-Verlag, Berlin, 2009.
\newblock Old and new.

\bibitem[Wen11]{Wenger-compactness}
Stefan Wenger.
\newblock Compactness for manifolds and integral currents with bounded diameter
  and volume.
\newblock {\em Calculus of Variations and Partial Differential Equations},
  40(3-4):423--448, 2011.

\end{thebibliography}
\end{document}